\documentclass{article}  
\usepackage[pdftex,
bookmarksnumbered,
bookmarksopen,
pagebackref,
colorlinks,
citecolor=blue,
linkcolor=blue,]{hyperref}
\usepackage{amsmath}
\usepackage{amssymb}
\usepackage{mathrsfs}
\usepackage{amsfonts}
\usepackage{amsthm}
\usepackage{algorithm}
\usepackage{booktabs}
\usepackage{listings}
\usepackage{boxedminipage}
\usepackage{algorithmic}
\usepackage{stmaryrd}
\usepackage{cite}
\usepackage{relsize}
\usepackage{lscape}
\usepackage[top=1.3in, bottom=1.5in, left=1.5in, right=1.5in]{geometry}
\usepackage{stmaryrd} 
\usepackage{relsize}%
\usepackage{url}
\usepackage{color,xcolor}

\newtheorem{theorem}{Theorem}[section]
\newtheorem{definition}{Definition}[section]
\newtheorem{proposition}{Proposition}[section]

\newtheorem{lemma}{Lemma}[section]
\newtheorem{corollary}{Corollary}[section]
\newtheorem{remark}{Remark}[section]

\newenvironment{mytabular}{\bgroup\tiny\tabular}{\endtabular\egroup}
\newenvironment{mytabular1}{\bgroup\footnotesize\tabular}{\endtabular\egroup}

    \newcommand{\T}{\ensuremath{ \mathbb R^{n_1\times \cdots\times n_d}   }}
    
    \newcommand{\bigxiaokuohao}[1]{\ensuremath{ \left(  #1 \right) }}

    \newcommand{\bigzhongkuohao}[1]{\ensuremath{ \left[   #1 \right] }}      
    
        \newcommand{\bigfnorm}[1]{\ensuremath{ \left\|   #1 \right\|_F }}    
                \newcommand{\bignorm}[1]{\ensuremath{ \left\|   #1 \right\|  }}        
                \newcommand{\bigllbracket}[1]{\ensuremath{ \left\llbracket   #1 \right\rrbracket }}      
                
                                \newcommand{\innerprod}[2]{\ensuremath{ \left\langle   #1 , #2\right\rangle }}      
                                
                                   \newcommand{\bigotimesu}{\ensuremath{     \bigotimes^d_{j=1} \mathbf u_{j,i  } }}    
                                   \newcommand{\bigotimesx}{\ensuremath{     \bigotimes^d_{j=1} \mathbf x_{j  } }}                                       
    \newcommand{\bigotimesy}{\ensuremath{     \bigotimes^d_{j=1} \mathbf y_{j  } }}  
                                       
       \newcommand{\gradAuji}{\ensuremath{       \mathcal A \bigotimes_{l\neq j}^d \mathbf u_{l,i}                                  }}

     \newcommand{\indicatorf}[2]{\ensuremath{  I_{C_{#1} }(#2)   } }          
     
          \newcommand{\deltak}[2]{\ensuremath{  \Delta^{#1,#2}   } }     
               
           \newcommand{\uomega}[1]{\ensuremath{  \mathbf u^{#1}_{j,i},\omega^{#1}_i   } }  

  \newcommand{\bigotimesuiterate}[1]{\ensuremath{     \bigotimes^d_{j=1} \mathbf u_{j,i  }^{#1} }}        
  
  \newcommand{\ball}[1]{\ensuremath{     \mathbb B_{#1}(\uomega{*})  }}

	\definecolor{darkgray}{rgb}{0.66, 0.66, 0.66}

\title{The Epsilon-Alternating Least Squares for Orthogonal Low-Rank Tensor Approximation and Its Global Convergence 
}

\author{Yuning Yang  \thanks{College of Mathematics and Information Science, Guangxi University, Nanning, 530004, China  (yyang@gxu.edu.cn).}                             
}

\begin{document} 
\maketitle

\begin{abstract}
 The epsilon alternating least squares ($\epsilon$-ALS) is developed and analyzed for canonical polyadic decomposition (approximation) of a higher-order tensor where one or more of the factor matrices are assumed to be columnwisely orthonormal. It is shown that the algorithm globally converges to a KKT point for all tensors without any assumption. For the   original ALS, by further studying the properties of the polar decomposition, we also establish its global convergence under a  reality assumption  not stronger than those in the literature. These results completely  address a question concerning the global convergence   raised in [L. Wang, M. T. Chu and B. Yu, \emph{SIAM J. Matrix Anal. Appl.}, 36 (2015), pp. 1--19].  In addition, an initialization procedure is proposed, which possesses a provable lower bound when the number of columnwisely orthonormal factors is one. Armed with this initialization procedure, numerical experiments show that the   $\epsilon$-ALS  exhibits a promising performance in terms of efficiency and effectiveness.
	
\end{abstract}

\hspace{2mm}\vspace{3mm}

\pagestyle{myheadings} \thispagestyle{plain} \markboth{Y. YANG}{$\epsilon$-ALS and  Convergence}

\section{Introduction}
Given a $d$-th  order tensor $\mathcal A$, the canonical polyadic decomposition (CPD) consists of decomposing $\mathcal A$ into a sum of component rank-1 tensors \cite{kolda2010tensor,comon2014tensors,cichocki2015tensor,Sidiropoulos2016ten}. In reality the decomposition is rarely exact, whereas approximation models and algorithms based on optimization are needed. However, due to degeneracy issues, the optimization problem of approximating CPD might not attain its optimum \cite{kolda2010tensor}. To overcome this, constraints such as orthogonality are imposed \cite{kolda2001orthogonal}. Typically, this requires one or more of the latent factors to be columnwisely orthonormal. Imposing such constraints not only makes the model more stable  but also reflects   real-world applications. For example, extracting the commonalities of a set of images can be formulated as a third-order CPD with two factors having   orthonormal columns \cite{shashua2001linear}. Other applications involve   DS-CDMA systems \cite{sorensen2010parafac} and independent component analysis \cite{comon1994independent}.

The  most widely used algorithm for CPD might be the alternating least squares (ALS), which updates the latent factor matrices sequentially, and is regarded as a ``workhorse'' algorithm \cite{kolda2010tensor}.  ALS methods have been tailored under various circumstances, especially in the presence  of latent columnwisely orthonormal factors \cite{sorensen2012canonical,wang2015orthogonal,chen2009tensor}. For orthogonal Tucker decomposition, 
various algorithms have been developed; see, e.g., \cite{ishteva2013jacobi,li2018globally,savas2010quasi,de2000a}. 

Throughout this work, we denote $t~(1\leq t\leq d)$ as the number of latent factors having orthonormal columns.  \cite{wang2015orthogonal} established the global convergence\footnote{By \emph{global convergence}, we mean that started from any initializer, the whole sequence generated by an algorithm converges to a single limit point which is a KKT point. There is no guarantee that the KKT point is a global optimizer.} of ALS for \emph{almost all} tensors when $t=1$, and conjectured that the global convergence still holds when $t>1$. \cite{guan2019numerical} studied   this problem by considering simultaneous updating two vectors corresponding to two non-orthogonal factors at a time, and showed that if certain matrices constructed from every limit point admit simple leading singular values or have full column rank, then  the global convergence for \emph{almost all} tensors holds.

The   results of \cite{guan2019numerical} conditionally address the question raised in \cite{wang2015orthogonal}. In this work, we intend to   address these issues as completely as possible in the following two senses: 1) to prove the global convergence for \emph{all} tensors for any $1\leq t\leq d$; 2) to prove the global convergence   \emph{without any  assumption}. Towards these goals, we have obtained the following results in this work, where the analysis is mainly based on a study of the polar decomposition  in which some useful information might be ignored by the existing literature:

1) An $\epsilon$-ALS   is developed. At each iterate, $\epsilon$-ALS imposes a very small perturbation to ALS, which shares the same computational complexity as the (unperturbed) ALS.  The perturbation parameters are free to choose. We prove that started from any initializer, the   $\epsilon$-ALS  globally converges to a KKT point of the problem for any $1\leq t\leq d$ without any assumption. Thus one can very safely use $\epsilon$-ALS to solve the problem in question.

2) We also analyze the convergence of the (unperturbed) ALS. It is shown that if there exists a limit point, where certain matrices constructed from this limit point are of full column rank, then this limit point is a KKT point and global convergence still holds ($1\leq t\leq d)$. The assumption makes sense and is not stronger than those used in the literature. To prove the convergence, some new ideas are developed, which can be regarded as a complement to   the existing frameworks of proving convergence and might be of independent interest.

3) In addition, an initialization procedure is proposed, by merging the   HOSVD \cite{de2000on} and tensor best rank-1 approximation \cite{he2010approximation}. A provable lower  bound is established when $t=1$. Armed with this procedure,   experimental results show some promising observations of the   $\epsilon$-ALS. 
 
To see the convergence results of the aforementioned algorithms more clearly, we briefly summarize them in Table \ref{tab:convergence_results}. Here `(a.a.) $  \checkmark$' means global convergence for almost all tensors, while `-' stands for there is no any assumption.
\begin{table}[htbp]
			\renewcommand\arraystretch{0.7}
	\centering
	\caption{\label{tab:convergence_results}Convergence results and assumptions of different methods, details discussed in Sect. \ref{sec:global_conv_results}.}
	\begin{mytabular}{cccccc}
		\toprule
		 &   $\epsilon$-ALS &  (unperturbed) ALS & \cite{chen2009tensor} & \cite{wang2015orthogonal} & \cite{guan2019numerical}   \\
		 &   {Alg. \ref{alg:main}, $ \small\epsilon_i>0$, $i=1,2$} &    {Alg. \ref{alg:main}, $ \epsilon_i=0$, $i=1,2$} & & & \\
		\midrule
	Problem  \eqref{prob:orth_main}  &   $1\leq t\leq d$  &  $1\leq t\leq d$ &  $ t=d$    &
  $t=1$ &   $1\leq t\leq d$   \\
		\midrule 
		Global  Convergence   & $\checkmark$  & $\checkmark$ & $\checkmark$  & (a.a.) $\checkmark$& (a.a.) $\checkmark$   \\
		\midrule
		Assumption & -  &  $\checkmark$   &  $\checkmark$   & -    &  $\checkmark$       \\
		\bottomrule
	\end{mytabular}%
	\label{tab:effect_tau3}%
\end{table}%

The rest of the paper is organized as follows. Sect. \ref{sec:problem} describes the problem under consideration and discuss some properties, with the   $\epsilon$-ALS given in Sect. \ref{sec:alg}. The convergence results and global convergence proofs are respectively presented in Sect. \ref{sec:convergence} and \ref{sec:Proof_global_covergence}. Numerical results are illustrated in Sect. \ref{sec:numer}. Sect. \ref{sec:conclusions} draws some conclusions.

%

\section{Problem Statement} \label{sec:problem}

Given a $d$-th order tensor $\mathcal A\in \T$, the CPD is to decompose $\mathcal A$ into a sum of rank-1 tensors:
\begin{equation}\label{prob:cpd}			      \setlength\abovedisplayskip{2pt}
\setlength\abovedisplayshortskip{2pt}
\setlength\belowdisplayskip{2pt}
\setlength\belowdisplayshortskip{2pt}
\mathcal A \approx  \sum^R_{i=1}\nolimits\sigma_i \mathcal T_i,~~{\rm with}~~		      \setlength\abovedisplayskip{2pt}
\setlength\abovedisplayshortskip{2pt}
\setlength\belowdisplayskip{2pt}
\setlength\belowdisplayshortskip{2pt}
\mathcal T_i:=\bigotimes^d_{j=1}\nolimits \mathbf u_{j,i  } := \mathbf u_{1,i}\otimes\cdots\otimes \mathbf u_{d,i} \in \T
\end{equation}
denoting the rank-1 terms, given by the outer product of $\mathbf u_{j,i}  \in\mathbb R^{n_j}  $, $1\leq j\leq d$; $\sigma_i$'s are real scalars; $R>0$ is a given integer.  Throughout this work, we denote matrix $U_j$ by stacking the $j$-th factor of the tensor  $\bigotimes^d_{j=1} \mathbf u_{j,i  }$,  $1\leq i\leq R$   together, namely,
\[			      \setlength\abovedisplayskip{2pt}
\setlength\abovedisplayshortskip{2pt}
\setlength\belowdisplayskip{2pt}
\setlength\belowdisplayshortskip{2pt}
U_j:= \bigzhongkuohao{ \mathbf u_{j,1},\ldots, \mathbf u_{j,R}} \in\mathbb R^{n_j\times R}.
\]

Problem \eqref{prob:cpd} can be formulated as minimizing the following   objective function:
\begin{equation}			      \setlength\abovedisplayskip{2pt}
\setlength\abovedisplayshortskip{2pt}
\setlength\belowdisplayskip{2pt}
\setlength\belowdisplayshortskip{2pt}
\label{prob:obj}
F(\mathbf u_{j,i},\sigma_i)=F(U_j,\sigma_i):=\frac{1}{2}\bigfnorm{\mathcal A- \sum^R_{i=1} \nolimits\sigma_i \mathcal T_i }^2:= \frac{1}{2}\bigfnorm{ \mathcal A -  \bigllbracket{ \boldsymbol{\sigma}; U_1,\ldots,U_d }      }^2,
\end{equation}
where $\|\cdot\|_F$ denotes the Frobenius norm of a tensor induced by the usual inner product between two tensors;   the last equation follows the notations of \cite{kolda2010tensor}; here   $\boldsymbol{\sigma}:= [ \sigma_1,\ldots,\sigma_R ]^\top$. Throughout this paper, we respectively denote $(\mathbf u_{j,i})$ and $(U_j)$ as the tuples of $\mathbf u_{j,i}$ and $U_j$, $1\leq j\leq d$, $1\leq i\leq R$. 

In this work, we consider the situation that one or more $U_j$'s are columnwisely orthonormal; without loss of generality, we assume that the last $t$ $U_j$'s are columnwisely orthogonal, namely, we assume that
\[			       \setlength\abovedisplayskip{3pt}
\setlength\abovedisplayshortskip{3pt}
\setlength\belowdisplayskip{3pt}
\setlength\belowdisplayshortskip{3pt}
U_j^\top U_j = I,~d-t+1 \leq j\leq d,~ 1\leq t\leq d,
\]
where $I\in\mathbb R^{R\times R}$ denotes the identity matrix. Under the above constraints,   $\mathbf u_{j,i}$, $d-t+1\leq j\leq d, 1\leq i\leq R$ are bounded. To overcome the scaling ambiguity of the first $(d-t)$ $U_j$, by noticing the scalars $\sigma_i$, we can without loss of generality normalize each column $\mathbf u_{j,i}$, such that $ \mathbf u_{j,i}^\top \mathbf u_{j,i}=1$, $ 1\leq j\leq d-t$, $ 1\leq i \leq R$. 
With the above notations, the problem we consider in this paper is formulated as the following optimization problem
 \begin{equation} \label{prob:orth_main}
 \setlength\abovedisplayskip{4pt}
 \setlength\abovedisplayshortskip{4pt}
 \setlength\belowdisplayskip{4pt}
 \setlength\belowdisplayshortskip{4pt}
 \begin{split}
 &\min~ F(\mathbf u_{j,i}, \sigma_i  )\\
 &~    {\rm s.t.}~  \mathbf u_{j,i}^\top \mathbf u_{j,i} =1, 1\leq j\leq d-t ,  1\leq i \leq R,\\
 & ~~~~~~ U_j^\top U_j = I,  d-t+1\leq j\leq d.
 \end{split}
 \end{equation}
 Clearly, the problem is meaningful only if $n_j\geq R$, $ d-t+1\leq j\leq d$.
\paragraph{Uniqueness of exact decomposition} Before diving further into the model, we first shortly discuss the essential uniqueness of decomposition of $\mathcal A$, that is, provided that  $\mathcal A = \llbracket \boldsymbol{\sigma}; U_1,\ldots,U_d\rrbracket$,   when  the solution to \eqref{prob:orth_main} can be unique. Here  by uniqueness, we mean that $\mathcal A = \sum^R_{i=1}\sigma_i\bigotimesu$ is the only possible way of combination of rank-1 terms, except for the elementary indeterminacy of permutation, namely, $( \mathbf u_{1,1},\ldots,\mathbf u_{d,R},\sigma_1,\ldots,\sigma_R)$ is the only solution to \eqref{prob:orth_main} except for permutation. In general, the well-known sufficient conditions for the uniqueness of decomposition of a third-order tensor are due to Kruskal \cite{kruskal1977three}, which have been generalized to tensors of higher-order ($d\geq 3$) \cite{sidiropoulos2000uniqueness}:
\begin{equation}			       \setlength\abovedisplayskip{3pt}
\setlength\abovedisplayshortskip{3pt}
\setlength\belowdisplayskip{3pt}
\setlength\belowdisplayshortskip{3pt}
\label{eq:uniqueness_k_rank_order_d}
\sum^d_{j=1}\nolimits \boldsymbol{k}_{U_j} \geq 2R+ (d-1),
\end{equation}
where $\boldsymbol{k}_{U}$ denotes the Kruskal rank of a matrix $U$, i.e., the maximum value $k$ such that any $k$ columns of $U$ are linearly independent. 
Based on \eqref{eq:uniqueness_k_rank_order_d} we have
\begin{proposition}
	\label{prop:uniqueness_determined}
	Given that $\mathcal A = \llbracket \boldsymbol{\sigma};U_1,\ldots, U_d\rrbracket$, $\sigma_i\neq 0$, $\mathbf u_{j,i}\neq 0$ for each $j$ and $i$, where $U_j$ $ d-t+1\leq j\leq d$ are columnwisely orthonormal. Assume that $R\geq 2$.   
\begin{enumerate}
	\item If $t=1$ and $\sum^{d-1}_{j=1}\boldsymbol{k}_{U_j} \geq R + d-1$, or
	\item if $t=2$ and there exists a $ \bar j\leq d-t$ such that $\boldsymbol{k}_{U_{\bar j}}\geq 2$, or
	\item if $3\leq t\leq d$, 
\end{enumerate}
then
	  the decomposition is unique.  
\end{proposition}
\begin{proof} It follows from the columnwise  orthonormality of $U_j$, $d-t+1\leq j\leq d$  that $\boldsymbol{k}_{U_j} = R$, $d-t+1\leq j\leq d$, and so
		\begin{equation*}			      \setlength\abovedisplayskip{3pt}
		\setlength\abovedisplayshortskip{3pt}
		\setlength\belowdisplayskip{3pt}
		\setlength\belowdisplayshortskip{3pt}
		\sum^d_{j=1}\nolimits\boldsymbol{k}_{U_j} - 2R - (d-1)  =    \sum^{d-t}_{j=1}\nolimits \boldsymbol{k}_{U_j} + tR - 2R - (d-1) .
	\end{equation*}
	The $t=1$ case directly follows. When $t=2$,  
 $\mathbf u_{j,i}\neq 0$ tells us that $\boldsymbol{k}_{U_j}\geq 1$ for each $j$, which gives
 \[			       \setlength\abovedisplayskip{3pt}
 \setlength\abovedisplayshortskip{3pt}
 \setlength\belowdisplayskip{3pt}
 \setlength\belowdisplayshortskip{3pt}
\sum^{d-t}_{j=1}\nolimits \boldsymbol{k}_{U_j} + tR - 2R - (d-1)
\geq d-3 + 2 +2 R - 2R - (d-1) = 0. 
 \]
    When $t\geq 3$, it follows from $R\geq 2$ that
    \[			       \setlength\abovedisplayskip{3pt}
    \setlength\abovedisplayshortskip{3pt}
    \setlength\belowdisplayskip{3pt}
    \setlength\belowdisplayshortskip{3pt}
 \sum^{d-t}_{j=1}\nolimits \boldsymbol{k}_{U_j} + tR - 2R - (d-1)   \geq  (d-t) + tR-2R - (d-1) \geq 0.
    \]
Thus the uniqueness of decomposition follows from \eqref{eq:uniqueness_k_rank_order_d}.
	\end{proof}

The above results are based on deterministic conditions, which show that the decomposition is more likely to be unique   as $t$ turns to be larger. In particular, the uniqueness property must hold when $t\geq 3$.
 In the following we present a simple generic condition guaranteeing the uniqueness when $t=1,2$. Recall that a tensor decomposition is called generic if it holds with probability one if every entry of the factor matrices are drawn from absolutely continuous probability density functions \cite{sorensen2012canonical}. For a generic matrix $U\in\mathbb R^{n\times R}$, it   holds that $\boldsymbol{k}_U ={\rm rank}(U)= \min\{n,R  \}$. The following can be derived.
\begin{corollary}
	\label{col:uniqueness_generic}
	Under the setting of Proposition \ref{prop:uniqueness_determined}, assume in addition that $U_j$ are generic and $n_j\geq 2$, $1\leq j\leq d-t$.  
	\begin{enumerate}
		\item If $t=1$ and $d\geq R+1 $, or
		\item if $t=2$ and $d\geq 5$, or
		\item if $t=1$ or $2$ and $n_j\geq R$, $1\leq j\leq d-t$, 
	\end{enumerate}
then the decomposition is generically unique.
	\end{corollary}
\begin{proof} Under the assumptions, note that $\boldsymbol{k}_{U_j}=\min\{n_j,R \}\geq 2$;   when $t=1$, 
	\[    \setlength\abovedisplayskip{3pt}
	\setlength\abovedisplayshortskip{3pt}
	\setlength\belowdisplayskip{3pt}
	\setlength\belowdisplayshortskip{3pt}
		\sum^d_{j=1}\nolimits\boldsymbol{k}_{U_j} - 2R - (d-1) \geq 2(d-1) +  R - 2R-(d-1) \geq 0.
	\]
	When $t=2$,
	\[    \setlength\abovedisplayskip{3pt}
	\setlength\abovedisplayshortskip{3pt}
	\setlength\belowdisplayskip{3pt}
	\setlength\belowdisplayshortskip{3pt}
\sum^d_{j=1}\nolimits\boldsymbol{k}_{U_j} - 2R - (d-1) \geq 2(d-2) +  2R - 2R-(d-1) \geq 0.
\]	
In particular, when $n_j\geq R$, $\boldsymbol{k}_{U_j}\geq R$, and the results follow naturally.
	\end{proof}

 \paragraph{Equivalent formulation and  attainability of the optimum}
 Due to  that one or more $U_j$ are columnwisely orthonormal and the normality of $\mathbf u_{j,i}$,   the tensors $\mathcal T_i$'s are   orthonormal as well,   i.e.,
 \[			      \setlength\abovedisplayskip{3pt}
 \setlength\abovedisplayshortskip{3pt}
 \setlength\belowdisplayskip{3pt}
 \setlength\belowdisplayshortskip{3pt}
 \innerprod{\mathcal T_{i_1}}{\mathcal T_{i_2}} = \delta_{i_1,i_2}, 1\leq i_1,i_2\leq R,
 \]
 where $\delta_{{i_1},{i_2}}$ represents the Kronecker delta. With this property, similar to  \cite{chen2009tensor,wang2015orthogonal}, it can be checked that \eqref{prob:orth_main} is equivalent to the maximization problem as follows
 \begin{eqnarray}			      \setlength\abovedisplayskip{3pt}
 \setlength\abovedisplayshortskip{3pt}
 \setlength\belowdisplayskip{3pt}
 \setlength\belowdisplayshortskip{3pt}
 \label{prob:ortho_main_max}
&&\max~ G(\mathbf u_{j,i}) = G(\mathbf u_{j,i},\sigma_i):= \sum^R_{i=1}\nolimits \sigma_i^2 \nonumber\\
&&~    {\rm s.t.}~  \sigma_i = \innerprod{\mathcal A}{\bigotimesu},\\
&& ~~~~~~  \mathbf u_{j,i}^\top \mathbf u_{j,i} =1,  1\leq j\leq d-t,  1\leq i\leq R,~U_j^\top U_j = I,  d-t+1\leq j\leq d.\nonumber
 \end{eqnarray}
 Since the constraints are bounded and the objective function is continuous, the optimum can be achieved.
 
 \paragraph{KKT system} 
 The KKT system of \eqref{prob:ortho_main_max} takes a similar fashion as \cite{chen2009tensor,wang2015orthogonal}. First we denote the notation $\mathcal A \bigotimes_{l\neq j}^d \mathbf u_{l,i} \in\mathbb R^{n_j}$ as the gradient of $\innerprod{\mathcal A}{\bigotimes^d_{l=1}\mathbf u_{l,i}}$ with respect to $\mathbf u_{j,i}$. For notational convenience we also denote $\mathbf v_{j,i}:= \mathcal A \bigotimes_{l\neq j}^d \mathbf u_{l,i} $. 
 By introducing dual variables $\eta_{j,i}\in\mathbb R$, $ 1\leq j\leq d-t$, $ 1\leq i\leq R$, and $\Lambda_j\in\mathbb R^{R\times R}$, $ d-t+1\leq j\leq d$, where $\Lambda_j$'s are symmetric matrices, we denote the Lagrangian function of the above problem as
 \begin{equation*} \setlength\abovedisplayskip{3pt}
 \setlength\abovedisplayshortskip{3pt}
 \setlength\belowdisplayskip{3pt}
 \setlength\belowdisplayshortskip{3pt}
 \label{prob:lagrangian_main_max}
 L( \mathbf u_{j,i}) = \sum^R_{i=1}\nolimits\sigma_i^2 - \sum_{j,i=1}^{d-t,R}\nolimits \eta_{j.i}\bigxiaokuohao{ \mathbf u_{j,i}^\top \mathbf u_{j,i} -1} - \sum^{d}_{j=d-t+1}\nolimits\innerprod{\Lambda_j}{ U_j^\top U_j - I}.
 \end{equation*}
 Taking derivative with respect to each $\mathbf u_{j,i}$  and setting it to zero yields the following KKT system
 \begin{small}
 \begin{equation}
 \label{prob:kkt_main_max_tmp} \setlength\abovedisplayskip{3pt}
 \setlength\abovedisplayshortskip{3pt}
 \setlength\belowdisplayskip{3pt}
 \setlength\belowdisplayshortskip{3pt}
 \left\{\begin{array}{ll}{\sigma_i \gradAuji =   \eta_{j,i} \mathbf u_{j,i}   }, & j=1,\ldots,d-t,~i=1,\ldots,R, \\ 
   \mathbf u_{j,i}^\top \mathbf u_{j,i} =1, &j=1,\ldots,d-t, i=1,\ldots,R,\\
 \sigma_i \gradAuji =   \sum^R_{r=1}  (\Lambda_j)_{i,r} \mathbf u_{j,r },& j=d-t+1,\ldots,d,~i = 1,\ldots,R,\\
   U_j^\top U_j = I, &j=d-t+1,\ldots,d,\\
   \sigma_i\in\mathbb R,~ \Lambda_j  \in\mathbb R^{R\times R},
 \end{array}\right.
 \end{equation}
 \end{small}
in which the third equality uses the symmetry of $\Lambda_j$, where   $(\Lambda_j)_{i,r}$ denotes the $(i,r)$-th entry of $\Lambda_j$. Recalling   $\mathbf v_{j,i}=\mathcal A \bigotimes_{l\neq j}^d \mathbf u_{l,i}$,   $\innerprod{\mathbf v_{j,i}}{\mathbf u_{j,i}} = \innerprod{\mathcal A}{\bigotimesu}$, it holds that $\eta_{j,i} = \sigma_i^2$ for each $i$ and $ 1\leq j\leq d-t$, and $   (\Lambda_j)_{i,i} = \sigma_i^2$ for each $i$ and $j=d-t+1,\ldots,d$. On the other hand, $(\Lambda_j)_{i,r} = \sigma_i\innerprod{\mathbf v_{j,i}}{\mathbf u_{j,r}}$ for $r\neq i$. Putting these back to \eqref{prob:kkt_main_max_tmp} gives
\begin{small}
  \begin{equation}
 \label{prob:kkt_main_max} \setlength\abovedisplayskip{3pt}
 \setlength\abovedisplayshortskip{3pt}
 \setlength\belowdisplayskip{3pt}
 \setlength\belowdisplayshortskip{3pt}
 \left\{\begin{array}{ll}{  \gradAuji =   \sigma_i \mathbf u_{j,i}   }, & j=1,\ldots,d-t,~i=1,\ldots,R, \\
  \mathbf u_{j,i}^\top \mathbf u_{j,i} =1, &j=1,\ldots,d-t, i=1,\ldots,R,\\ 
 \sigma_i \gradAuji =   \sum^R_{r\neq i}    (\Lambda_j)_{i,r} \mathbf u_{j,r} + \sigma_i^2 \mathbf u_{j,i},& j=d-t+1,\ldots,d,~i = 1,\ldots,R,\\
 U_j^\top U_j = I, &j=d-t+1,\ldots,d,\\
    \sigma_i\in\mathbb R,~ \Lambda_j  \in\mathbb R^{R\times R}.
 \end{array}\right.
 \end{equation}
 \end{small}

\section{Algorithm}\label{sec:alg}

 \begin{algorithm}[h] 
 	\algsetup{linenosize=\tiny}
\footnotesize
 	\caption{The $\epsilon$-ALS for solving \eqref{prob:orth_main}	\label{alg:main}}
 	\begin{algorithmic}[1]
 		\REQUIRE $U_j^0 = [\ldots,\mathbf u_{j,i}^0,\ldots]$, $j=1,\ldots,d $, with $\|\mathbf u_{j,i}\|=1$, $1\leq j\leq d-t$, $1\leq i\leq R$; $(U^0)^\top U^0_j = I$, $d-t+1\leq j\leq d$; $\boldsymbol{ \omega}^0 = [\ldots,\omega^0_i,\ldots]^\top\in\mathbb R^R$,  $\omega_i^0 = \sigma^0_i/\sqrt{\sum^R_{i=1}(\sigma^0_i)^2  } $,  $\sigma_i = \innerprod{\mathcal A}{\bigotimesuiterate{0}}$; $\epsilon_1,\epsilon_2\geq 0$.
 		\FOR{$k=0,1,\ldots,$}
 			\FOR{$j = 1,2,\ldots,d-t$} 
 		    	\FOR{$i = 1,2,\ldots,R$}
 					\STATE  $\mathbf v_{j,i}^{k+1} =  {\mathcal A}{\mathbf u_{1,i}^{k+1}\otimes \cdots\otimes \mathbf u_{j-1,i}^{k+1} \otimes \mathbf u_{j+1,i}^k \otimes\cdots\otimes \mathbf u_{d,i}^k  }$ \color{gray}\% {lines 4 and 5 can be done simultaneously for all $i=1,\ldots,R$}\color{black}
 					\STATE $\mathbf u^{k+1}_{j,i} = \frac{ \tilde{\mathbf v}_{j,i}^{k+1}  }{ \|\tilde{\mathbf v}_{j,i}^{k+1} \| }$, where $\tilde{\mathbf v}^{k+1}_{j,i} = \mathbf v^{k+1}_{j,i}\cdot\omega^k_i + \color{black} \epsilon_1\cdot\mathbf u^k_{j,i}$ \color{black}
 				\ENDFOR
 			\ENDFOR{ {~~~~~~~~~~~~~~~~~~~~~~~~~~\color{gray}\% end of the updating of non-orthonormal constraints}}\color{black}
			\FOR{$j = d-t+1, \ldots,d $}
				\FOR{$i = 1,2,\ldots,R$}
					\STATE    $\mathbf v^{k+1}_{j,i}= {\mathcal A}{\mathbf u_{1,i}^{k+1}\otimes \cdots\otimes \mathbf u_{j-1,i}^{k+1} \otimes \mathbf u_{j+1,i}^k \otimes\cdots\otimes \mathbf u_{d,i}^k  }$ \color{gray}\% lines 10 and 11 can be done simultaneously  for all $i=1,\ldots,R$\color{black}
					\STATE $\tilde{\mathbf v}_{j,i}^{k+1} = \mathbf v^{k+1}_{j,i}\cdot\omega^k_i + \color{black}  \epsilon_2\cdot \mathbf u_{j,i}^k$ \color{black}
				\ENDFOR
								\STATE $\tilde V^{k+1}_j = [  \tilde{\mathbf v}_{j,1}^{k+1},\ldots,\tilde{\mathbf v}_{j,R}^{k+1}   ]$
				\STATE $[U^{k+1}_j,H^{k+1}_j] = {\rm polar\_decomp}(\tilde V^{k+1}_j)$~~~~~~~~~~~~\color{gray}\% polar decomposition of $V^{k+1}$
			\ENDFOR{~~~~~~~~~~~~~~~~~~~~~~~~~~\color{gray}\% end of the updating of orthonormal constraints}\color{black} \color{black}
			\STATE  $\sigma_i^{k+1} = \innerprod{\mathcal A}{\bigotimesuiterate{k+1}}$, $\boldsymbol{ \sigma}^{k+1}=[\sigma^{k+1}_1,\ldots,\sigma^{k+1}_R]^\top$, $\boldsymbol{ \omega}^{k+1} = \frac{ \boldsymbol{ \sigma}^{k+1}  }{ \bignorm{\boldsymbol{ \sigma}^{k+1}}  }$.
 		\ENDFOR
 	\end{algorithmic}
 \end{algorithm}
Algorithm \ref{alg:main}  is designed for     \eqref{prob:ortho_main_max}. Note that the supersceipt $(\cdot)^k$ means the $k$-th iterate, while the subscript $(\cdot)_{j,i}$ represents the $i$-th column of the $j$-th factor matrix, $1\leq i\leq R$, $1\leq j\leq d$; these notations are different from \cite{chen2009tensor,wang2015orthogonal}. ${\mathcal A}{\mathbf u_{1,i}^{k+1}\otimes \cdots\otimes \mathbf u_{j-1,i}^{k+1} \otimes \mathbf u_{j+1,i}^k \otimes\cdots\otimes \mathbf u_{d,i}^k  }$ represents the gradient of $\innerprod{\mathcal A}{\bigotimes^d_{l=1}\mathbf u_{l,i} }$ with respect to $\mathbf u_{j,i}$ at the point $(\mathbf u^{k+1}_{1,i},\ldots,\mathbf u^{k+1}_{j-1,i},\mathbf u^{k}_{j,i},\ldots,\mathbf u^k_{d,i}  )$.

The algorithm can be divided into two parts: lines 2--7 is devoted to the computation of the first $(d-t)$ $U_j$, while lines 8--15 amounts to finding the last $t$ factor matrices. lines 2--7 is similar to the scheme of the  higher-order power method \cite{delathauwer2000on} and \cite{wang2015orthogonal}, whereas in lines 8--15, after the updating of $\tilde V$, a polar decomposition is needed to ensure the columnwise orthonormality of $U_j$. The polar decomposition with its property will be studied in the next section. It can be easily seen that $\epsilon$-ALS shares the same computational complexity as ALS \cite{wang2015orthogonal}.

When $\epsilon_1,\epsilon_2=0$    and $t=1$, $\epsilon$-ALS is closely related to the ALS proposed in \cite{sorensen2012canonical,wang2015orthogonal}, except for the  slight differences on updating $\boldsymbol{\omega}$; when $t=d$, $\epsilon$-ALS is exactly the same as \cite[Algorithm 1]{chen2009tensor}.
 
 When $\epsilon_1,\epsilon_2>0$, this can be recognized as imposing shift terms, which is a well-known technique in the computation of matrix and tensor eigenvalues \cite{kolda2011shifted}. However, the difference is that $\epsilon$ here can be arbitrarily small, and numerical results indicate  that a sufficiently small   $\epsilon$  indeed exhibits  better performance. This is why we prefer to call the method $\epsilon$-ALS instead of shifted ALS.  Imposing $\epsilon$ terms can help in avoiding assumptions in   the convergence analysis, as will be studied later.
 


 \subsection{Initializer}
 In the following, we present a procedure to obtain an initializer for problem \eqref{prob:ortho_main_max}, trying to capture as much of dominant information from the data tensor $\mathcal A$ as possible.
 
  \begin{boxedminipage}{0.97\textwidth}\small
  	\begin{equation}  \label{proc:init}
\noindent {\rm Procedure}~ (\mathbf u_{1,1},\ldots,\mathbf u_{d,R}) =  {\rm get\_initializer}(\mathcal A) 
\end{equation}
 
 	1. For each $j=d-t+1,\ldots,d$, compute the left leading $R$ singular vectors of the unfolding matrix $A_{(j)}\in\mathbb R^{n_j\times \prod^{d}_{l\neq j} n_l}$, denoted as $(\mathbf u_{j,1},\ldots,\mathbf u_{j,R} )$.
 	
 	2. For each $i=1,\ldots,R$, compute a rank-1 approximation solution to the tensor $\mathcal A\times_{d-t+1,i} \mathbf u_{d-t+1}^\top\times \cdots \times_d\mathbf u_{d,i}^\top\in \mathbb R^{n_1\times \cdots\times n_{d-t}}$:
 	\[
 	(\mathbf u_{1,i},\ldots,\mathbf u_{d-t,i}) = {\rm rank1approx}( \mathcal A\times_{d-t+1} \mathbf u_{d-t+1}^\top\times \cdots \times_d\mathbf u_{d,i}^\top   )
 	\]
 	
 	3. Return $(\mathbf u_{1,1},\ldots, \mathbf u_{d,R})$.
 \end{boxedminipage}

Here for the definitions of unfolding $A_{(j)}$ and tensor-vector product $\mathcal A\times_j \mathbf u_j^\top \in \mathbb R^{n_1\times \cdots\times n_{j-1}\times n_{j+1}\times \cdots \times n_d}$, one can refer to \cite{kolda2010tensor}.

The first step of the above procedure is a part of the truncated HOSVD \cite{de2000a}, in that only $t$ columnwisely orthonormal matrices are computed by means of matrix SVD. Given $U_{d-t+1},\ldots, U_d$,  a resonable way to obtain  $U_1,\ldots, U_{d-t}$ is to maximize $G(\mathbf u_{j,i})$ of \eqref{prob:ortho_main_max} with respect to     $\mathbf u_{1,1},\ldots,\mathbf u_{{d-t},R}$, namely, 
\begin{eqnarray}
\label{prob:rank1approx}			      \setlength\abovedisplayskip{2pt}
\setlength\abovedisplayshortskip{2pt}
\setlength\belowdisplayskip{2pt}
\setlength\belowdisplayshortskip{2pt}
&&\max~   \sum^R_{i=1} \innerprod{\mathcal A}{\bigotimesu}^2 = \sum^R_{i=1} \innerprod{ \mathcal A\times_{d-t+1} \mathbf u_{d-t+1}^\top\times \cdots \times_d\mathbf u_{d,i}^\top  }{ \bigotimes^{d-t}_{j=1}\nolimits\mathbf u_{j,i}   }^2 \nonumber\\
&& ~~~~~~  \mathbf u_{j,i}^\top \mathbf u_{j,i} =1, j=1,\ldots,d-t, i=1,\ldots,R.
\end{eqnarray}
Since the objective function and the constraints    are decoupled with respect to each $\mathbf u_{j,i}$, it amounts to solving $R$ separate tensor best rank-1 approximation problems of size $n_1\times\cdots\times n_{d-t}$. This is the explanation of step 2. However, it is known that such a problem is NP-hard in general \cite{hillar2013most}. Nevertheless, approximation solution methods exist; see, e.g., \cite{de2000on,he2010approximation}. In the following we provide a procedure to compute a rank-1 approximation solution to a $\mathcal B\in\mathbb R^{n_1\times\cdots\times n_m}$ in the same spirit of \cite{he2010approximation} but with slightly better efficiency. The procedure is defined recursively, in which \texttt{reshape}$(\cdot)$ is the same as that in  Matlab.

  \begin{boxedminipage}{0.97\textwidth}\small
	\begin{equation}  \label{proc:rank1approx}
	\noindent {\rm Procedure}~ (\mathbf x_1,\ldots,\mathbf x_m) =  {\rm rank1approx}(\mathcal B) 
	\end{equation}
	
	1. If $m=1$, return $\mathbf x_1=\mathcal B/\bignorm{B}_F$.
	
	2. If $m=2$, return $(\mathbf x_1,\mathbf x_2)$ as the normalized leading singular vector pair of $\mathcal B\in\mathbb R^{n_1\times n_2}$.
	
	3. Reshape $\mathcal B$ as $B = $\texttt{reshape}$(\mathcal B,\prod^{m-2}_{j=1}n_j,n_{m-1}n_m ) \in \mathbb R^{\prod^{m-2}_{j=1}n_j \times n_{m-1}n_m}$. 
	
	Compute $(\mathbf x_{1,\ldots, m-2 },\mathbf x_{ m-1,m})\in\mathbb R^{\prod^{m-2}_{j=1}n_j }\times \mathbb R^{n_{m-1}n_m}$ as the normalized leading singular vector pair of the matrix $B$.
	
	4. Reshape $X_{ \mathbf xm-1,m}:=$\texttt{reshape}$(\mathbf x_{ m-1,m }, n_{m-1},n_m)\in\mathbb R^{n_{m-1}\times n_m}$; compute 
	
	$~~~~~~~~~~~~~~~~~~~~~~~~~~~~~~(\mathbf x_{m-1},\mathbf x_m)=$rank1approx($X_{m-1,m}$).
	
	5. Denote $\mathcal X_{1,\ldots,m-2}:=\mathcal B\times_{m-1}\mathbf x_{m-1}^\top\times_m\mathbf x_m^\top\in\mathbb R^{n_1\times\cdots\times n_{m-2}}$; compute 
	
	$~~~~~~~~~~~~~~~~~~~~~~~~~~~~~~(\mathbf x_1,\ldots,\mathbf x_{m-2})=$rank1approx($\mathcal X_{1,\ldots, m-2}$).

	6. Return $(\mathbf x_1,\ldots,\mathbf x_m)$.
\end{boxedminipage}

 To analyze the performance of Procedure \ref{proc:init}, the following property is useful.
\begin{proposition}\label{prop:rank1}			       \setlength\abovedisplayskip{3pt}
	\setlength\abovedisplayshortskip{3pt}
	\setlength\belowdisplayskip{3pt}
	\setlength\belowdisplayshortskip{3pt}
	Let $\mathcal B\in\mathbb R^{n_1\times\cdots\times n_m}$ with $m\geq 3$ and $n_1\leq \cdots\leq n_m$. Let $(\mathbf x_1,\ldots,\mathbf x_m)$ be generated by Procedure \ref{proc:rank1approx}. Then it holds that 	
	\begin{equation}\label{eq:rank1approx}\setlength\abovedisplayskip{4pt}
	\setlength\abovedisplayshortskip{4pt}
	\setlength\belowdisplayskip{4pt}
	\setlength\belowdisplayshortskip{4pt}
	\begin{split}
&	 \innerprod{\mathcal B}{\bigotimes^m_{j=1}\nolimits\mathbf x_j}  
	\geq  {\bignorm{   B  }_2}/\xi(m)  \\
&~~~~~~~~~~~~~~~~~~\,\geq \bignorm{B}_F/(\xi(m) \sqrt{n_{m-1}n_m} ),
	\end{split}
	\end{equation}
	where $B$ is defined in the procedure, $\xi(m) = \sqrt{ \prod^{m-1}_{j=1}n_j\cdot \prod^{m/2-2}_{j=1}n_{2j+1} \cdot n_2^{-1}  }$ if $m$ is even, and when $m$ is odd, $\xi(m) = \sqrt{\prod^{m-1}_{j=2}n_j\cdot\prod^{(m+1)/2-2}_{j=1}n_{2j}}$. 
	\end{proposition}
\begin{proof}
	Denote $\|\mathcal X\|_2:= \max_{\|\mathbf y_j\|=1  }\innerprod{\mathcal X}{\bigotimes^m_{j=1} \mathbf y_j}$ as the spectral norm of a tensor $\mathcal X$.
	Without loss of generality we only prove the case when  $m$ is even, because otherwise the tensor can be understood as in the space $\mathbb R^{1\times n_1\times\cdots\times n_m}$, whose order is again even. Write $m=2p$. We prove the claim inductively on $p$. When $p=1$, the claim clearly holds. Assume that it holds for $p>1$. When the order is $ 2(p+1)$, denote $\mathcal X_{1,\ldots,2p}=\mathcal B\times_{2p+1}\mathbf x_{2p+1}^\top\times_{2p+2}\mathbf x_{2p+2}^\top$ as that in the procedure, and the matrix $X_{1,\ldots,2p  }:=$\texttt{reshape}$(\mathcal X_{1,\ldots,2p},\prod^{2p-2}_{j=1}n_j, n_{2p-1}n_{2p}) \in\mathbb R^{\prod^{2p-2}_{j=1}n_j\times n_{2p-1}n_{2p}}$ accordingly; denote $\mathcal X_{\mathbf x1,\ldots,2p }:= $\texttt{reshape}$(\mathbf x_{1,\ldots,2p },n_1,\ldots,n_{2p }) \in\mathbb R^{n_1\times\cdots\times n_{2p }}$. On the other hand, from step 3 we see that $B^\top\mathbf x_{1,\ldots, 2p } = \bignorm{B}_2 \mathbf x_{2p+1 ,2p+2}$. We then have
	\begin{eqnarray}
	\innerprod{\mathcal B}{\bigotimes^{2p+2}_{j=1}\nolimits\mathbf x_{j}} &=& \innerprod{\mathcal X_{1,\ldots,2p}}{\bigotimes^{2p}_{j=1}\nolimits\mathbf x_j} \label{eq:prooff:1} \\
	&\geq & \bignorm{  X_{1,\ldots,2p}}_2/ \xi( 2p)     \geq \bignorm{   X_{1,\ldots, 2p}  }_F/(\xi( 2p) \sqrt{n_{2p-1}n_{2p}})  \label{eq:prooff:2} \\
	& =& \max_{ \bignorm{\mathcal Y}_F=1  } \innerprod{\mathcal X_{1,\ldots,2p}}{\mathcal Y} /(\xi(2p) \sqrt{n_{2p-1}n_{2p}} ) \label{eq:prooff:3}\nonumber\\
	&  \geq&  \innerprod{\mathcal X_{1,\ldots,2p}}{ \mathcal X_{\mathbf x1,\ldots,2p} }/({\xi(2p)}\sqrt{n_{2p-1}n_{2p}}  )   \label{eq:prooff:4}\\
	&=& \bignorm{B}_2 \innerprod{X_{\mathbf x2p+1,2p+2}}{\mathbf x_{2p+1}\otimes\mathbf x_{2p+2}}/({\xi(2p)}\sqrt{n_{2p-1}n_{2p}}  )   \label{eq:prooff:5}\\
	&\geq& \bignorm{B}_2 /( \xi(2p) \sqrt{n_{2p-1}n_{2p}n_{2p+1}} ) = \bignorm{B}_2/\xi(2p+2), \label{eq:prooff:6} 
	\end{eqnarray}
	where \eqref{eq:prooff:1} uses the definition of $\mathcal X_{1,\ldots,2p}$; the first inequality of \eqref{eq:prooff:2} is based on the induction, while the second one relies on the relation between matrix spectral norm and Frobenius norm; \eqref{eq:prooff:4} follows from the definition of $\mathcal X_{\mathbf x 1,\ldots,2p}$ and that $\|\mathbf x_{1,\ldots,2p}\|=1$; \eqref{eq:prooff:5} is due to   
%
  step 3 of the procedure, and the definition of $X_{\mathbf x2p+1,2p+2}$ in step 4; \eqref{eq:prooff:6} comes from step 4 and that $\|X_{\mathbf x2p+1,2p+2}\|_F=1$, and the definition of $\xi(2p)$. The second inequality of \eqref{eq:rank1approx} again follows from the relation between matrix spectral and Frobenius norms. This completes the proof.
	\end{proof}

When $t=1$, we present a provable lower bound for the initializer generated by Procedure \ref{proc:init}. 
Assume that the singular values of the unfolding matrix $A_{(d)}$ satisfy $\lambda_1\geq \cdots\geq \lambda_R\geq \cdots$ with $\lambda_R>0$.
\begin{proposition}
	\label{prop:initializer}
	Let $(\mathbf u_{j,i})$ be generated by Procedure \ref{proc:init}. Then it holds that
\begin{equation}\label{eq:init}  \setlength\abovedisplayskip{3pt}
\setlength\abovedisplayshortskip{3pt}
\setlength\belowdisplayskip{3pt}
\setlength\belowdisplayshortskip{3pt}
G(\mathbf u_{j,i})	=\sum^R_{i=1}\nolimits \innerprod{\mathcal A}{\bigotimesu}^2 \geq \frac{\sum^R_{i=1}\lambda_i^2}{\xi(d-1)^2n_{d-2}n_{d-1}} \geq \frac{G_{\max}}{\xi(d-1)^2n_{d-2}n_{d-1}} ,
\end{equation}
where $\xi(\cdot)$ is defined in Proposition \ref{prop:rank1}, and $G_{\max}$ denotes the maximal value of \eqref{prob:ortho_main_max}.
\end{proposition}
\begin{proof}
	Let $(\mathbf u_{d,i},\mathbf v_{d,i})$ be the singular vector pair corresponding to $\lambda_i$ of $A_{(d)}$. Then it holds that $ A_{(d)}^\top\mathbf u_{d,i} = \lambda_i \mathbf v_{d,i} $. Denote   $\mathcal V_{d,i} = $\texttt{reshape}$(\mathbf v_{d,i}, n_1,\ldots,n_{d-1} ) \in \mathbb R^{n_1\times \cdots \times n_{d-1}}$. From step 2 of Procedure \ref{proc:init} we see that $(\mathbf u_{1,i},\ldots,\mathbf u_{{d-1},i})$ is obtained from applying Procedure \ref{proc:rank1approx} to $\mathcal A\times_d\mathbf u_{d,i}^\top=\lambda_i\mathcal V_{d,i}$.  We have
	\begin{eqnarray*} \setlength\abovedisplayskip{3pt}
		\setlength\abovedisplayshortskip{3pt}
		\setlength\belowdisplayskip{3pt}
		\setlength\belowdisplayshortskip{3pt}
	\sum^R_{i=1}\nolimits \innerprod{\mathcal A}{\bigotimesu}^2 &= & \sum^R_{i=1}\nolimits\lambda_i^2 \innerprod{\mathcal V_{d,i}}{\bigotimes^{d-1}_{j=1}\nolimits \mathbf u_{j,i}  }^2\\
	&\geq& \sum^R_{i=1}\nolimits \lambda_i^2/  (\xi(d-1)^2n_{d-2}n_{d-1}),
	\end{eqnarray*}
where the inequality follows from   Proposition \ref{prop:rank1} and by noticing that $\bignorm{\mathcal V_{d,i}}_F=1$. The second inequality of \eqref{eq:init} is due to that $\sum^R_{i=1}\lambda_i^2$ is clearly an upper bound of problem \eqref{prob:ortho_main_max}.
\end{proof}
\begin{remark}
1. Procedure \ref{proc:init}    is in the same spirit, while generalizes  the initialization procedure for best rank-1 approximation; see, e.g., \cite[Sect. 6]{kofidis2002best} and \cite[Alg. 1]{he2010approximation}. The most important information has been captured as much as possible.

2. While it is not the main concern, the lower bound can be improved with  more careful analysis. Empirically we have observed that the   procedure has a better ratio $G(\mathbf u_{j,i})/\sum^R_{i=1}\lambda_i^2$ 
  than the theoretical one   in \eqref{eq:init}, and  performs much better than randomized initializers with respect to the objective value of \eqref{prob:ortho_main_max}, when $1\leq t\leq d-1$, as illustrated in Table \ref{tab:init}.
  
  3. The analysis cannot be directly extended to $t>2$, unless the mode-$d$ unfolding   of $\mathcal A\times_{d-t+1}U_{d-t+1}^\top\times \cdots \times_d U_d^\top$ is equivalent to   $A_{(d)} ( U_d\odot\cdots \odot U_{d-t+1}  ) $, where $\odot$ denotes the Khatri-Rao product \cite{kolda2010tensor}.
\end{remark}

\begin{table}[h]
	\centering
			\renewcommand\arraystretch{0.2}
		\caption{\label{tab:init}Illustration of the ratio $G(  {\mathbf u_{j,i}^p   })/ G(\mathbf u_{j,i}^r   )$, where $\mathbf u^p_{j,i}$ are generated by Procedure \ref{proc:init}, while $\mathbf u_{j,i}^r$ are generated randomly. The table shows the averaged ratios on $n\times n\times n\times n$     tensors over $50$ instances  in which the entries are randomly drawn from the normal distribution. We see that except $t=4$, the procedure gives a much better initializer.}
	\begin{mytabular1}{ccccc}
		\toprule
		& $n=10$ & $n=20$& $n=30$ & $n=40$ \\
		\toprule
	$t=1$	& 80.31 &  188.98&  237.95& 347.87 \\
	\midrule
	$t=2$	&  66.31&  134.11&  199.87&  302.24\\
	\midrule
	$t=3$	&  19.87&  34.25& 54.81 &69.87\\
	\midrule
	$t=4$	& 2.17 & 1.95& 1.97& 1.6 \\
	\bottomrule
	\end{mytabular1}
\end{table}

 \section{Convergence Results of $\epsilon$-ALS}\label{sec:convergence}
 
 \subsection{Properties of the polar decomposition}
 ~\\
 \begin{theorem}[Relation between polar decomposition and SVD, c.f. \cite{higham1986computing}] \label{th:polar_dec}
Let $C\in\mathbb R^{m\times n}$, $m\geq n$. Then there exist 	 $U\in\mathbb R^{m\times n}$ and a unique positive semidefinite matrix $H\in\mathbb R^{n\times n}$ such that
\[ \setlength\abovedisplayskip{3pt}
\setlength\abovedisplayshortskip{3pt}
\setlength\belowdisplayskip{3pt}
\setlength\belowdisplayshortskip{3pt}
C = UH,~U^\top U = I\in\mathbb R^{n\times n}.
\]
$(U,H)$ is the polar decomposition of $C$. If ${\rm rank}(C)=n$, then $H$ is positive definite and $U$ is uniquely determined.

Furthermore, let $H=Q\Lambda Q^\top$, $Q,\Lambda\in\mathbb R^{n\times n}$ be the eigenvalue decomposition of $H$, namely, $Q^\top Q = QQ^\top = I$, $\Lambda = {\rm diag}(\lambda_1,\ldots,\lambda_n)$ be a diagonal matrix where $\lambda_1\geq \cdots\geq \lambda_n\geq 0$. Then   
$
U = PQ^\top$, and  $C = P\Lambda Q^\top $
  is exactly the SVD of $C$.
 \end{theorem}
 
The following property holds \cite{higham1986computing,chen2009tensor}. Although \cite{higham1986computing,chen2009tensor} do not mention the converse part, from the relation between SVD and polar decomposition, we can not hard to check   the validness of the converse part.
 \begin{lemma}\label{lem:polar_max}
 	Let $C\in\mathbf R^{m\times n}$, $m\geq n$ admit the polar decomposition $C = UH$ with $U\in\mathbb R^{m\times n},H\in\mathbb R^{n\times n}$, $U^\top U = I$, and $H$ is a positive semidefinite matrix. Then
 	\[ \setlength\abovedisplayskip{3pt}
 	\setlength\abovedisplayshortskip{3pt}
 	\setlength\belowdisplayskip{3pt}
 	\setlength\belowdisplayshortskip{3pt}
 	U\in\arg\max_{X^\top X = I} {\rm tr}\bigxiaokuohao{  X^\top C   }.
 	\]
 	Conversely, if $U$ is a maximizer of the above problem, then there exists a positive semidefinite matrix $H$ such that $C=UH$.
 \end{lemma}
 
  The following result strengthens the above lemma which characterizes an explicit  lower bound for the gap between ${\rm tr}(U^\top C) - {\rm tr}(X^\top C)$.  For that purpose, 
 for $C\in\mathbb R^{m\times n}$, $m\geq n$, let $C = P\Lambda Q^\top$ denote the SVD of $C$, where $P\in\mathbb R^{m\times n}$, $\Lambda,Q\in\mathbb R^{n\times n}$, $P^\top P = I$, $Q^\top Q=QQ^\top = I$, $\Lambda = {\rm diag}(\lambda_1,\ldots,\lambda_n)$ where $\lambda_i$'s are arranged in a descending order.  

 \begin{lemma}
\label{lem:polar_max_sufficient}
Under the setting of Lemma \ref{lem:polar_max}, let $X\in\mathbb R^{m\times n}$ be an arbitrary matrix satisfying $X^\top X = I$. Then there holds
\[			      \setlength\abovedisplayskip{3pt}
\setlength\abovedisplayshortskip{3pt}
\setlength\belowdisplayskip{3pt}
\setlength\belowdisplayshortskip{3pt}
{\rm tr}(U^\top C) - {\rm tr}(X^\top C) \geq \frac{\lambda_n}{2}\| U - X\|_F^2.
\]
 \end{lemma}
\begin{proof}
From Theorem \ref{th:polar_dec}, it holds that $U = PQ^\top$ where $P$ and $Q$ are described as above. For an arbitrary matrix $X$, define $Y:= XQ\in\mathbf R^{m\times n}$. Then it is clear that $Y^\top Y = I$. On the other hand, from the orthogonality of $Q$ it follows $X = YQ^\top$. Write $P = [\mathbf p_1,\ldots,\mathbf p_n]$, $Q = [\mathbf q_1,\ldots,\mathbf q_n]$; then SVD shows that $C\mathbf q_i = \lambda_i \mathbf p_i$.  Correspondingly we also write $Y =[\mathbf y_1,\ldots,\mathbf y_n]$ with $\mathbf y_i\in\mathbb R^m$.    With these pieces at hand, we have
\begin{eqnarray*}			      \setlength\abovedisplayskip{1pt}
	\setlength\abovedisplayshortskip{1pt}
	\setlength\belowdisplayskip{1pt}
	\setlength\belowdisplayshortskip{1pt}
{\rm tr}(U^\top C) - {\rm tr}(X^\top C) &=&   {\rm tr}( QP^\top C) - {\rm tr}( QY^\top C)={\rm tr}( P^\top CQ) - {\rm tr}( Y^\top CQ)\\
&=&  \sum^n_{i=1}\mathbf p_i^\top C\mathbf q_i - \sum^n_{i=1} \mathbf y_i^\top C\mathbf q_i\\
&=&  \sum^n_{i=1}\lambda_i \mathbf p_i^\top \mathbf p_i - \sum^n_{i=1} \lambda_i\mathbf y_i^\top \mathbf p_i\\
&=&   \sum^n_{i=1}\frac{\lambda_i}{2}\| \mathbf p_i - \mathbf y_i\|^2\\
&\geq & \frac{\lambda_n}{2}\| P-Y\|_F^2 =\frac{\lambda_n}{2}\| U - X\|_F^2,
\end{eqnarray*}
where the fifth equality uses the fact that $\mathbf p_i$ and $\mathbf y_i$ are normalized, and the last equality is due to the orthogonality of $Q$. The proof has been completed.
\end{proof}


When $\lambda_n>0$, we see that the inequality in the above lemma holds strictly. Since $\lambda_n>0$ is equivalent to ${\rm rank}(C)=n$, this also shows   why $U$ is uniquely determined, as stated in Theorem \ref{th:polar_dec}.
The following lemma, which generalizes Lemma \ref{lem:polar_max_sufficient}, also serves as a key property to prove the convergence.
\begin{lemma}
	\label{lem:polar_max_augmented_sufficient}
	Let $X\in\mathbb R^{m\times n}$ be an arbitrary matrix satisfying $X^\top X = I$ where $m\geq n$. For any $C\in\mathbb R^{m\times n}$, consider the polar decomposition of $\tilde C := C+\epsilon X = UH$ with $U^\top U = I$ and $H$ being positive semidefinite where $\epsilon>0$. Denote $\lambda_i(\cdot)$ as the $i$-th largest singular value of a matrix. Then there holds
	\[			       \setlength\abovedisplayskip{3pt}
	\setlength\abovedisplayshortskip{3pt}
	\setlength\belowdisplayskip{3pt}
	\setlength\belowdisplayshortskip{3pt}
	{\rm tr}(U^\top C) - {\rm tr}(X^\top C) \geq \frac{\lambda_n( \tilde C ) + \epsilon}{2} \| U-X\|_F^2.
	\]
\end{lemma}
\begin{proof}
 $\|U\|_F^2 =  \|X \|_F^2= n^2$ from their columnwise  orthornormality and so
\begin{eqnarray*} \setlength\abovedisplayskip{3pt}
	\setlength\abovedisplayshortskip{3pt}
	\setlength\belowdisplayskip{3pt}
	\setlength\belowdisplayshortskip{3pt}
	{\rm tr}(U^\top C ) - {\rm tr}(X^\top C) & = & \innerprod{U}{C } + \frac{ \epsilon}{2}\|U\|_F^2 - \innerprod{X}{C }- \frac{\epsilon}{2}\|X\|_F^2\\
	&=& \innerprod{U-X}{\tilde C} + \frac{\epsilon}{2}\|U-X\|_F^2\\
	&\geq & \frac{\lambda_n(\tilde C) + \epsilon  }{2}\|U-X\|_F^2,
\end{eqnarray*}
where the second equality uses the Taylor expansion of the quadratic function $q(Z) = \innerprod{Z}{C} + \frac{\epsilon}{2}\|Z\|_F^2$ at $X$, while the inequality is due to the assumption and Lemma \ref{lem:polar_max_sufficient}. This completes the proof.
\end{proof}

\setlength\abovedisplayskip{3pt}
\setlength\abovedisplayshortskip{3pt}
\setlength\belowdisplayskip{3pt}
\setlength\belowdisplayshortskip{3pt}
    \subsection{Diminishing of  $\| \mathbf u^k_{j,i}- \mathbf u^{k+1}_{j,i}\| $ when $\epsilon_i>0$}
An important issue to prove the convergence is to claim   $\| \mathbf u^k_{j,i}- \mathbf u^{k+1}_{j,i}\|\rightarrow 0$ as $k\rightarrow \infty$ for each $j$ and $i$. To achieve this,
  it suffices  to show that 
$
 h(\mathbf u^{k+1}_{j,i}) - h(\mathbf u^{k }_{j,i}) \geq c\|\mathbf u^k_{j,i}- \mathbf u^{k+1}_{j,i}\|^2
$
 for a fixed constant $c>0$ and for some potential cost function $h(\cdot)$. 
 By noticing the     auxiliary variables $\omega_i$   in Algorithm \ref{alg:main} and recall that $ \sigma_i = \innerprod{\mathcal A}{\bigotimesu}$ in \eqref{prob:ortho_main_max} ,  we propose to use the following cost function
\begin{equation*} \setlength\abovedisplayskip{3pt}
\setlength\abovedisplayshortskip{3pt}
\setlength\belowdisplayskip{3pt}
\setlength\belowdisplayshortskip{3pt}
\label{prob:cost_fun_potential}
H(\mathbf u_{i,j} ,\omega_i  )  := \sum^R_{i=1} \nolimits\sigma_i\omega_i. 
\end{equation*}
Based on the simple observation that the optimal solution of $\max_{\|\boldsymbol{ \omega}\|=1} \innerprod{\boldsymbol{ \omega}}{\mathbf a} $  is given by 
	$\boldsymbol{ \omega} = \mathbf a/\|\mathbf a\|$,   we have 
\begin{proposition}
	\label{prop:connection_sigma_omega}
	Consider 
	 \begin{equation}
	 \label{prob:ortho_main_max_omega} \setlength\abovedisplayskip{3pt}
	 \setlength\abovedisplayshortskip{3pt}
	 \setlength\belowdisplayskip{3pt}
	 \setlength\belowdisplayshortskip{3pt}
	\begin{split}
	&\max~ H(\mathbf u_{i,j},\omega_i) =  \sum^R_{i=1}{   \omega_i}{\innerprod{\mathcal A}{\bigotimesu}} \\
	&~    {\rm s.t.}~   \mathbf u_{j,i}^\top \mathbf u_{j,i} =1, j=1,\ldots,d-t, i=1,\ldots,R,~U_j^\top U_j = I, j=d-t+1,\ldots,d,\\
	&~~~~~~ \|\boldsymbol{ \omega}  \|=1.
	\end{split}
	\end{equation}
	Then problems \eqref{prob:ortho_main_max} and \eqref{prob:ortho_main_max_omega} share the same KKT points with respect to $\mathbf u_{j,i}$. Moreover, the optimal value of \eqref{prob:ortho_main_max_omega} is the square root of that of \eqref{prob:ortho_main_max}.
\end{proposition}
\begin{proof}
The KKT system of \eqref{prob:ortho_main_max_omega} is given by
\begin{small}
 \begin{equation}
\label{prob:kkt_main_max_omega}\setlength\abovedisplayskip{3pt}
\setlength\abovedisplayshortskip{3pt}
\setlength\belowdisplayskip{3pt}
\setlength\belowdisplayshortskip{3pt}
\left\{\begin{array}{ll}{ \gradAuji =   \sigma_i \mathbf u_{j,i}   }, & j=1,\ldots,d-t,~i=1,\ldots,R, \\ 
\mathbf u_{j,i}^\top \mathbf u_{j,i} =1, &j=1,\ldots,d-t, i=1,\ldots,R,\\
\omega_i \gradAuji =   \sum^R_{r=1}  (\Lambda_j)_{i,r} \mathbf u_{j,r },& j=d-t+1,\ldots,d,~i = 1,\ldots,R,\\
U_j^\top U_j = I, &j=d-t+1,\ldots,d,\\
\sigma_i\in\mathbb R,~ \Lambda_j  \in\mathbb R^{R\times R},\omega_i = \sigma_i/\sqrt{\sum^R_{i=1}\sigma_i^2}.
\end{array}\right.
\end{equation}
\end{small}
Thus \eqref{prob:kkt_main_max_omega} is the same as \eqref{prob:kkt_main_max_tmp}   with respect to $\mathbf u_{j,i}$    and the second claim is clear.
\end{proof}

The cost function $H(\cdot)$ serves as a guidance to the diminishing properties of successive difference of the variables. Moreover, although   $H(\cdot)$ looks quite simple, 
 it is linear with respect to each variable. This gives    flexibility to the convergence proof and makes our analysis   different from the previous work \cite{chen2009tensor,wang2015orthogonal,guan2019numerical}. With $H(\cdot)$ at hand, in what follows we estimate the decreasing of $\mathbf u_{j,i}^k - \mathbf u_{j,i}^{k+1}$ when $\epsilon_i>0$. 

\paragraph{Diminishing of $\| \mathbf u^k_{j,i}- \mathbf u^{k+1}_{j,i}\| $ where $1\leq j\leq d-t$}
For each $1\leq j\leq d-t$ and $1\leq i\leq R$, recalling the definition of $\mathbf v^{k+1}_{j,i}, \mathbf u^{k+1}_{j,i}$ in   Algorithm \ref{alg:main}, we have that at the $(k+1)$-th iterate,
\begin{eqnarray*} \setlength\abovedisplayskip{3pt}
	\setlength\abovedisplayshortskip{3pt}
	\setlength\belowdisplayskip{3pt}
	\setlength\belowdisplayshortskip{3pt}
&&  \omega^k_i \innerprod{\mathbf v_{j,i}^{k+1}}{\mathbf u_{j,i}^{k+1}} - \omega^k_i\innerprod{\mathbf v_{j,i}^{k+1}}{\mathbf u^k_{j,i}}\\  
&=&   \omega^k_i\innerprod{\mathbf v_{j,i}^{k+1}}{\mathbf u_{j,i}^{k+1}} + \frac{\epsilon_1}{2}\|\mathbf u^{k+1}\|^2 -\omega^k_i \innerprod{\mathbf v_{j,i}^{k+1}}{\mathbf u^k_{j,i}}  - \frac{\epsilon_1}{2}\|\mathbf u^{k}\|^2 \nonumber \\
  &=& \innerprod{\mathbf v_{j,i}^{k+1}\omega^k_i + \epsilon_1\mathbf u^k_{j,i}  }{\mathbf u_{j,i}^{k+1} - \mathbf u^k_{j,i}} + \frac{\epsilon_1}{2}\|\mathbf u^{k+1}_{j,i} - \mathbf u^k_{j,i}\|^2\nonumber \\
&=&  { \|\mathbf v^{k+1}_{j,i}\omega^k_i + \epsilon_1\mathbf u^k_{j,i}\|  - \innerprod{\mathbf v^{k+1}_{j,i}\omega^k_i + \epsilon_1\mathbf u^k_{j,i}}{\mathbf u^k_{j,i}}  } + \frac{\epsilon_1}{2}\|\mathbf u^{k+1}_{j,i} - \mathbf u^k_{j,i}\|^2 \nonumber\\
&=&\frac{  \|\mathbf v^{k+1}_{j,i}\omega^k_i + \epsilon_1\mathbf u^k_{j,i}\|}{2} \bigxiaokuohao{ 2- 2 \innerprod{  \frac{\mathbf v^{k+1}_{j,i}\omega^k_i + \epsilon_1\mathbf u^k_{j,i}   }{\|\mathbf v^{k+1}_{j,i}\omega^k_i + \epsilon_1\mathbf u^k_{j,i}\|  }  }{ \mathbf u^k_{j,i} }   } + \frac{\epsilon_1}{2}\|\mathbf u^{k+1}_{j,i} - \mathbf u^k_{j,i}\|^2 \nonumber\\
&=& \frac{  \|\mathbf v^{k+1}_{j,i}\omega^k_i + \epsilon_1\mathbf u^k_{j,i}\|+\epsilon_1}{2}\bignorm{ \mathbf u^{k+1}_{j,i} - \mathbf u^k_{j,i}  }^2,   \nonumber
\end{eqnarray*}
 where the first equality uses the normality of $\mathbf u^k_{j,i}$ for each $k$, the second one follows from the Taylor expansion of the quadratic function $q(\mathbf y) = \innerprod{\mathbf v^{k+1}_{j,i}\omega^k_i+ \epsilon_1\mathbf u^k_{j,i}}{\mathbf y} +\frac{\epsilon_1}{2}\|\mathbf y\|^2$ at $\mathbf u^k_{j,i}$, while the other ones are based on the definition of $\mathbf u^{k+1}_{j,i}$ in the algorithm. On the other hand, by noticing the definition of $H(\cdot)$, we get
 \begin{small}
\begin{eqnarray}  \setlength\abovedisplayskip{3pt}
\setlength\abovedisplayshortskip{3pt}
\setlength\belowdisplayskip{3pt}
\setlength\belowdisplayshortskip{3pt}
&&H(\mathbf u^{k+1}_{1,1}, \ldots, \mathbf u^{k+1}_{j,i-1},\mathbf u^{k+1}_{j,i},\mathbf u^{k }_{j,i+1},\ldots,\mathbf u^k_{d,R},\omega^k_i) - H(\mathbf u^{k+1}_{1,1}, \ldots, \mathbf u^{k+1}_{j,i-1},\mathbf u^{k }_{j,i },\mathbf u^{k }_{j,i+1 },\ldots,\mathbf u^k_{d,R},\omega^k_i)\nonumber\\
&&= \frac{  \|\mathbf v^{k+1}_{j,i}\omega^k_i + \epsilon_1\mathbf u^k_{j,i}\|+\epsilon_1}{2}\bignorm{ \mathbf u^{k+1}_{j,i} - \mathbf u^k_{j,i}  }^2,~\forall k.  \label{eq:descent_normalization}
\end{eqnarray}
\end{small}

\paragraph{Diminishing of $\| \mathbf u^k_{j,i}- \mathbf u^{k+1}_{j,i}\| $ where $d-t+1\leq j\leq d $}   From the definition of polar decomposition, we shall estimate $\|U^{k+1}_j - U^k_j\|_F$. Denote the diagonal matrix $\Omega^k:= {\rm diag}(\omega^k_1,\ldots,\omega^k_R)\in\mathbb R^{R\times R}$. Then $\tilde V^{k+1}_j$ in line 13 of Algorithm \ref{alg:main} can be written as $\tilde V^{k+1}_j = V^{k+1}_j\Omega^k + \epsilon U^k_j$, where $V^{k+1}_j = [ \mathbf v^{k+1}_{j,1},\ldots,\mathbf v^{k+1}_{j,R} ]$. By the definition of $U^{k+1}_j$ in the algorithm, we have
\begin{equation*}
U^{k+1}_jH^{k+1}_j =\tilde V^{k+1}_j  =  {V^{k+1}_j\Omega^k + \epsilon_2 U^k_j}.
\end{equation*}
Since $(U^k_j)^\top U^k_j = I$ for each $k$,     by setting $C=   V^{k+1}_j\Omega^k$,  $X= U^k_j$ and $U = U^{k+1}_j$  in Lemma \ref{lem:polar_max_augmented_sufficient}, we have
\[			       \setlength\abovedisplayskip{3pt}
\setlength\abovedisplayshortskip{3pt}
\setlength\belowdisplayskip{3pt}
\setlength\belowdisplayshortskip{3pt}
\innerprod{U^{k+1}_j}{  V^{k+1}_j\Omega^k} - \innerprod{U^k}{  V^{k+1}_j\Omega^k } \geq \frac{\lambda_R(\tilde V^{k+1}_j ) + \epsilon_1 }{2}\| U^{k+1}_j-U^k_j\|_F^2,
\]
where $\lambda_R(\tilde V^{k+1}_j)\geq 0$ is the $R$-th largest singular value of $\tilde V^{k+1}_j$. On the other hand, note that $ \innerprod{U_j}{  V^{k+1}_j\Omega^k} =H(\mathbf u^{k+1}_{1,1},\ldots,\mathbf u^{k+1}_{d-t,R},U^{k+1}_{d-t+1},\ldots,U^{k+1}_{j-1},U_j,U^{k}_{j+1},\ldots,U^{k }_d,\omega^k_i) $;  we thus have that for each $d-t+1\leq j\leq d$,
\begin{small}
\begin{eqnarray}
&& H(\mathbf u^{k+1}_{1,1}\ldots,\mathbf u^{k+1}_{d-t,R},U^{k+1}_{d-t+1}, \ldots,   U^{k+1}_{j}, U^{k}_{j+1},  \ldots,U^{k}_d,\omega^k_i) \nonumber\\
&&~~~~~~~~~~~~~~~~~~~~ -   H(\mathbf u^{k+1}_{1,1}\ldots,\mathbf u^{k+1}_{d-t,R},U^{k+1}_{d-t+1}, \ldots, U^{k+1}_{j-1}, U^{k }_{j},   \ldots,U^{k}_d,\omega^k_i) \nonumber\\
&\geq &\frac{\lambda_R(\tilde V^{k+1}_j ) + \epsilon_1 }{2}\| U^{k+1}_j-U^k_j\|_F^2 \label{eq:descent_polar}, ~\forall k.
\end{eqnarray}
\end{small}
Finally, after the updating of all the $\mathbf u_{j,i}$'s, by noticing line 16 of Algorithm \ref{alg:main},  we have, similar to \eqref{eq:descent_normalization}
\begin{equation}
\label{eq:descent_omega} \setlength\abovedisplayskip{3pt}
\setlength\abovedisplayshortskip{3pt}
\setlength\belowdisplayskip{3pt}
\setlength\belowdisplayshortskip{3pt}
H(\mathbf u^{k+1}_{j,i},\boldsymbol{ \omega}^{k+1}) - H(\mathbf u^{k+1}_{j,i},
\boldsymbol{\omega}^k) = \frac{\bignorm{\boldsymbol{ \sigma}^{k+1}}}{2}\bignorm{ \boldsymbol{ \omega}^{k+1} - \boldsymbol{ \omega}^k  }^2,~\forall k.
\end{equation}

Under the above discussions, we see that the sequence of the objective function is nondecreasing, whose limit is denoted as $H^{\infty}$   in the following. Since the constraints are compact, $H^{\infty}<\infty$. On the other hand, we can without loss of generality assume the initial objective value is positive.
\begin{proposition}
	\label{prop:nondecreasing}
	Given $\epsilon_i$ in Algorithm \ref{alg:main} with $\epsilon_i\geq 0$, $i=1,2$,	there holds
	\begin{eqnarray*} \setlength\abovedisplayskip{3pt}
		\setlength\abovedisplayshortskip{3pt}
		\setlength\belowdisplayskip{3pt}
		\setlength\belowdisplayshortskip{3pt}
		&&\infty>H^{\infty}\geq \cdots\geq	H(\mathbf u^{k+1}_{1,1},\ldots,\mathbf u^{k+1}_{d,R},\boldsymbol{ \omega}^{k+1}) \geq 	H(\mathbf u^{k+1}_{1,1},\ldots,\mathbf u^{k+1}_{d,R},\boldsymbol{ \omega}^{k}) \\
		&\geq&  \cdots \geq H(\mathbf u^{k+1}_{1,1}\ldots,\mathbf u^{k+1}_{d-t,R},U^{k+1}_{d-t+1}, \ldots,   U^{k+1}_{j}, U^{k}_{j+1},  \ldots,U^{k}_d,\omega^k_i) \\
		&\geq &\cdots \geq H(\mathbf u^{k+1}_{1,1}, \ldots, \mathbf u^{k+1}_{j,i},\mathbf u^{k }_{j,i+1},\ldots,\mathbf u^k_{d-t,R},  U^k_{d-t+1},\ldots, U^k_d, \omega^k_i)\\
		&\geq& \cdots\geq H(\mathbf u^{k}_{1,1},\ldots,\mathbf u^{k}_{d,R},\boldsymbol{ \omega}^{k}) \geq \cdots \geq H(\mathbf u^{0}_{1,1},\ldots,\mathbf u^{0}_{d,R},\boldsymbol{ \omega}^{0}) >0.
	\end{eqnarray*}
\end{proposition}

Next, given a small but fixed scalar $\epsilon_0>0$ such that $\epsilon_1,\epsilon_2\geq \epsilon_0>0$, by combining \eqref{eq:descent_normalization}, \eqref{eq:descent_polar} and \eqref{eq:descent_omega} together, we arrive at our goal
\begin{theorem}
	\label{th:sufficient_dec_positive_epsilon}
	Let $\{\uomega{k}  \}$ be generated by Algorithm \ref{alg:main} started from any initializer $(\uomega{0} )$ with $\epsilon_1,\epsilon_2\geq \epsilon_0>0$. Denote $c=\sqrt{\sum^R_{i=1}(\sigma^0_i)^2 } = H(\mathbf u^0_{j,i},\omega^0_i)>0$. 
	Then there holds that for any $k$,
	\begin{small}
\begin{eqnarray*}
 &&	H(\uomega{k+1}) - H(\uomega{k}) \geq \frac{\epsilon_0}{2}\sum^d_{j=1}\sum^R_{i=1}\|\mathbf u^{k+1}_{j,i} - \mathbf u^k_{j,i}\|^2 + \frac{c}{2}\|\boldsymbol{ \omega}^{k+1}-\boldsymbol{ \omega}^k\|^2.
\end{eqnarray*}
\end{small}
\end{theorem}
The case that $\epsilon_1=\epsilon_2=0$ needs more discussions in the next subsection.

\subsection{Diminishing of  $\| \mathbf u^k_{j,i}- \mathbf u^{k+1}_{j,i}\| $ when $\epsilon_i=0$}


First we consider $\epsilon_1=0$. In this case, we have to lower bound the coefficient $\|\mathbf v^{k+1}_{j,i}\omega^k_i + \epsilon_1\mathbf u^k_{j,i}\| = \bigfnorm{ \mathbf v^{k+1}_{j,i} \omega^k_i}$. Recall that when $\epsilon_1=0$, by the definition of $\mathbf u^{k+1}_{j,i}$, we have for each $i$ and for $1\leq j\leq d-t$,
\begin{small}
\[ 
\|\mathbf v^{k+1}_{j,i}\| = \innerprod{\mathbf v^{k+1}_{j,i}}{\mathbf u^{k+1}_{j,i}} = \mathcal A\otimes\mathbf u^{k+1}_{1,i}\otimes\cdots\otimes\mathbf u^{k+1}_{j,i}\otimes\mathbf u^k_{j+1,i }\otimes\cdots\otimes\mathbf u^k_{d,i} \geq \cdots\geq\mathcal A \bigotimesuiterate{k} = \sigma_i^k.
\] 
\end{small}
By Proposition \ref{prop:nondecreasing} and the definition of $\boldsymbol{ \omega}^k$, we have $\sum^R_{i=1}(\sigma^k_i)^2\leq    \sum^R_{i=1}(\sigma^{k+1}_i)^2\leq \cdots \leq (H^\infty)^2< \infty$. Thus
\begin{equation}\label{eq:lower_bound_v_omega}			      
\bigfnorm{ \mathbf v^{k+1}_{j,i} \omega^k_i }\geq (\sigma^k_i)^2/H^\infty.
\end{equation}
It remains to estimate $|\sigma^k_i|$. If it  is uniformly bounded away from zero, then we are done. Otherwise, now the trouble is that the nondecreasing property of $\{ \sum^R_{i=1}(\sigma^k_i)^2 \}$  does not necessarily imply that each 
  sequence $\{ \sigma^k_{i} \}_{k=0}^\infty$ is also nondecreasing. This is due to the way of updating    $U_j,d-t+1\leq j\leq d$ that computes $  \mathbf u_{j,1} \ldots,\mathbf u_{j,R}$ simultaneously.
 Nevertheless, we     discuss that this will not be an issue. Assume that, for example,   there is a subsequence   $\{| \sigma^{k_l}_1 | \}_{l=1}^\infty\rightarrow 0$. Passing to the subsequence if necessary, we assume that $\{\mathbf u^{k_l}_{j,i}  \}_{l=1}^\infty \rightarrow \mathbf u^*_{j,i}$. Thus $\sigma^*_1 = \innerprod{\mathcal A}{\bigotimes_{j=1}^d\mathbf u^*_{j,1}}=0$. Recall that the objective function is given by $\sum^R_{i=1}\omega_i\sigma_i$. Thus $\sigma^*_1=0$ means that the variables $\mathbf u_{j,1},1\leq  j\leq d$     contributes nothing to the objective function, and can   be considerably removed. In this case, it is possible that the parameter $R$ is chosen too large, and one should reduce it as $R\leftarrow R-1$.
    As a result, one should first reduce $R$ such that the degeneracy at $\sigma_i$ does not occur. In this sense, we can always assume that $\{|\sigma^k_i| \}_{k=0}^\infty$ is uniformly lower bounded away from zero for all $i$ by a positive constant $c_1>0$, and so \eqref{eq:lower_bound_v_omega} together with \eqref{eq:descent_normalization} yields
    \begin{small}
    \begin{eqnarray} \setlength\abovedisplayskip{3pt}
    \setlength\abovedisplayshortskip{3pt}
    \setlength\belowdisplayskip{3pt}
    \setlength\belowdisplayshortskip{3pt}
    &&H(\mathbf u^{k+1}_{1,1}, \ldots, \mathbf u^{k+1}_{j,i-1},\mathbf u^{k+1}_{j,i},\mathbf u^{k }_{j,i+1},\ldots,\mathbf u^k_{d-t,R},\ldots,\mathbf u^k_{d,R},\omega^k_i)\nonumber\\
    &&~~~~~~~~~~~~~~~~~~~~  - H(\mathbf u^{k+1}_{1,1}, \ldots, \mathbf u^{k+1}_{j,i-1},\mathbf u^{k }_{j,i },\mathbf u^{k }_{j,i+1 },\ldots,\mathbf u^k_{d-t,R},\ldots,\mathbf u^k_{d,R},\omega^k_i)\nonumber\\
    &=& \frac{  c_1 }{2}\bignorm{ \mathbf u^{k+1}_{j,i} - \mathbf u^k_{j,i}  }^2,~\forall k.  \label{eq:descent_normalization_e1is0}
	    \end{eqnarray}
	    \end{small}
  
We then consider $\epsilon_2=0$. In this case, $\tilde V^k_j = V^k_j$ in the algorithm, and the coefficient in \eqref{eq:descent_polar} is $\lambda_R(V^{k+1}_j)$, the $R$-th largest singular value of $V^{k+1}_j$. The issue now is to lower bound this value.  Note that $\lambda_R(V^{k+1}_j)>0$ is equivalent to that $V^{k+1}_j$ has full column rank. This is the assumption made in \cite[Theorem 5.7]{chen2009tensor}. In what follows, we wish to weaken such an assumption. Since $\{\uomega{k}  \}_{k=0}^\infty$ is bounded, limit points exist.  Assume that $(\uomega{*}  )$ is a limit point; define $V^*_j$ following the same vein  as $V^k_j$ in Algorithm \ref{alg:main}. In addition, denote 
$$			      \setlength\abovedisplayskip{3pt}
\setlength\abovedisplayshortskip{3pt}
\setlength\belowdisplayskip{3pt}
\setlength\belowdisplayshortskip{3pt} \mathbb B_\alpha(\mathbf u^*_{j,i},\omega^*_i) := \{ (\mathbf u_{j,i},\omega_i) \mid \bignorm{ (\mathbf u_{j,i},\omega_i)  -(\mathbf u^*_{j,i},\omega^*_i) }\leq \alpha   \}.$$
We first present a lemma that will be used in the following and in Section \ref{sec:proof:global_conv_zero_e}.
\begin{lemma}
	\label{lem:sufficient_dec_zero_e_in_a_ball}
		Let $\{\uomega{k}  \}$ be generated by Algorithm \ref{alg:main} started from any initializer $(\uomega{0} )$ with $\epsilon_1=\epsilon_2=  0$. Assume that there is a limit point $(\uomega{*})$, such that $V^*_j$ has full column rank, $j=d-t+1,\ldots,d$. Then there exists constants $\alpha_0>0$ and $c_2>0$, such that for $\forall  \bar k$ with $(\uomega{\bar k}) \in \ball{\alpha_0}$,    we have that for all $j = d-t+1,\ldots,d$,
		\begin{small}
\begin{eqnarray*}			      \setlength\abovedisplayskip{3pt}
	\setlength\abovedisplayshortskip{3pt}
	\setlength\belowdisplayskip{3pt}
	\setlength\belowdisplayshortskip{3pt}
&&	H( U^{\bar k+1}_1,\ldots, U^{\bar k+1}_{j-1},U^{\bar k+1}_{j} , U^{\bar k }_{j+1}, \ldots, U^{\bar k}_d,\omega^{\bar k}_i)    - H( U^{\bar k}_1,\ldots,U^{\bar k+1}_{j-1}, U^{\bar k}_j,U^{\bar k}_{j+1},\ldots, U^{\bar k}_d,\omega^{\bar k}_i  ) \\
& \geq& \frac{c_2}{4}\bigfnorm{  U^{\bar k+1}_j - U^{\bar k}_j  }^2.
\end{eqnarray*}
\end{small}
\end{lemma}
\begin{proof}
	The proof uses a similar argument as that of \cite[Theorem 3.2]{yang2016convergence}. Since $V^*_j$'s are of full column rank, we can assume that there is a positive constant $c_2>0$ such that
$\lambda_R(V^*_j)\geq c_2>0$ for  $j=d-t+1,\ldots,d$.   By the definition of $V_j$, there must exist a positive number $\bar\alpha>0$ and a ball $\mathbb B_{\bar\alpha}(\mathbf u^*_{j,i},\omega^*_i) $, such that for all $(\mathbf u_{j,i},\omega_i) \in \mathbb B_{\bar\alpha}(\mathbf u^*_{j,i},\omega^*_i)$, 
$\lambda_R( V_j ) \geq c_2/2>0$. In the following,  we write $(\mathbf u_{j,i},\omega_i) = ( \mathbf u_{1,1},\ldots,\mathbf u_{d,R},\omega_i ) := (U_1,\ldots,U_d,\omega_i)$.

 Let $\alpha_0$ and $\hat k$ be such that
\begin{equation}\label{eq:proof:0}			   
\alpha_0 < \frac{\bar\alpha}{2},~{\rm and}~ H^\infty-H(\uomega{k})    < \frac{\alpha_0^2c_2}{4t},~\forall~ k\geq \hat k .
\end{equation}  
Such $\alpha_0$ is well-defined because of the nondecreasing property of the objective value.  Since $(\uomega{*})$ is a limit point, the existence of   $\bar k$ with  $(\uomega{\bar k}) \in \ball{\alpha_0}$ makes sense. From \eqref{eq:descent_normalization_e1is0}, $\bignorm{\mathbf u^{k+1}_{j,i} - \mathbf u^k_{j,i}}\rightarrow 0$, $1\leq j\leq d-t$; so without loss of generality we   assume that $(U^{\bar k+1}_1,\ldots,U^{\bar k+1}_{d-t}, U^{\bar k}_{d-t+1},\ldots,U^{\bar k}_{d},\omega^{\bar k}_i  ) \in \ball{\alpha_0}$ as well, otherwise we can increase $\bar k$ or decrease $\alpha_0$ until it holds.
 In the following, we inductively show that if $(\uomega{\bar k})\in \ball{\alpha_0}$, then for all $j=d-t+1,\ldots,d$, 
\begin{equation}\label{eq:proof:00}			     
 (  U^{\bar k+1}_{ 1},\ldots,  U^{\bar k+1}_{j-1}, U^{\bar k}_{j },\ldots,U^{\bar k}_{d},\omega^{\bar k}_i  ) \in \ball{\bar\alpha},
\end{equation}
based on which we can obtain the inequality in question. Now $j=d-t+1$ clearly holds by the discussions above. Assume that it holds for $j=d-t+1,\ldots,m<d$. According to the definitions of $\ball{\bar\alpha}$ and $V^{\bar k+1}_j$, we have  $\lambda_R(V^{\bar k+1}_j) \geq \frac{c_2}{2}$. It thus similar to \eqref{eq:descent_polar} that for $j=d-t+1,\ldots,m$,
\begin{eqnarray}
\label{eq:proof:000} \setlength\abovedisplayskip{3pt}
\setlength\abovedisplayshortskip{3pt}
\setlength\belowdisplayskip{3pt}
\setlength\belowdisplayshortskip{3pt}
&&\frac{c_2}{4}\bignorm{ U^{\bar k+1}_j - U^{\bar k}_j }^2_F \leq     H( U^{\bar k+1}_1,\ldots, U^{\bar k+1}_{m-1} , U^{\bar k+1}_m,U^{\bar k}_{m+1},\ldots, U^{\bar k}_d,\omega^{\bar k}_i) \nonumber\\
	&& ~~~~~~~~~~~~~~~~~~~~~~~~~~~~~~- H( U^{\bar k}_1,\ldots,U^{\bar k+1}_{m-1}, U^{\bar k}_m,\ldots, U^{\bar k}_d,\omega^{\bar k}_i  )   .
\end{eqnarray}
In the following we verify \eqref{eq:proof:00} when $j=m+1$. We have
\begin{eqnarray*} \setlength\abovedisplayskip{3pt}
	\setlength\abovedisplayshortskip{3pt}
	\setlength\belowdisplayskip{3pt}
	\setlength\belowdisplayshortskip{3pt}
&&\bignorm{  ( U^{\bar k+1}_1,\ldots, U^{\bar k+1}_m, U^{\bar k}_{m+1},\ldots, U^{\bar k}_d,\omega^{\bar k}_i   ) - (\uomega{*})    }_F \\
&\leq& \bignorm{( U^{\bar k+1}_1,\ldots, U^{\bar k+1}_m, U^{\bar k}_{m+1},\ldots, U^{\bar k}_d,\omega^{\bar k}_i   )  -   ( U^{\bar k+1}_1,\ldots, U^{\bar k+1}_{d-t}, U^{\bar k}_{d-t+1},\ldots, U^{\bar k}_d,\omega^{\bar k}_i        )}_F \\
&&~~~~~~~~~~~~~~~~+ \bignorm{    ( U^{\bar k+1}_1,\ldots, U^{\bar k+1}_{d-t}, U^{\bar k}_{d-t+1},\ldots, U^{\bar k}_d,\omega^{\bar k}_i   ) - (\uomega{*})   }_F\\
&\leq& \sqrt{ \sum_{j=d-t+1}^{m } \bignorm{ U^{\bar k+1}_j - U^{\bar k}_j  }^2_F  }+ \alpha_0\\
&\leq& \sqrt{  \frac{4t}{c_2}\bigxiaokuohao{ H^\infty - H(\uomega{k})   }  } + \alpha_0 < \bar\alpha,
\end{eqnarray*}
where the second inequality is due to \eqref{eq:proof:000} and Proposition \ref{prop:nondecreasing}, and the last one obeys \eqref{eq:proof:0}. Thus induction method tells us that for $j=d-t+1,\ldots,d$, \eqref{eq:proof:00} holds. On the other hand, \eqref{eq:proof:0} together with  the definition of $\ball{\bar \alpha}$ and \eqref{eq:descent_polar} gives that \eqref{eq:proof:000} holds for all 
$j=d-t+1,\ldots,d$, provided that $(\uomega{\bar k})\in\ball{\alpha_0}$. The proof has been completed.
\end{proof}

Combining the above lemma with \eqref{eq:descent_normalization_e1is0} and \eqref{eq:descent_omega} (note that the latter two hold  for all $k$), we have the following:
\begin{theorem}
	\label{th:sufficient_dec_epsilon0}
	Under the setting of Lemma \ref{lem:sufficient_dec_zero_e_in_a_ball},   there exists constants $\alpha_0>0$ and $c_3>0$, such that for $\forall  \bar k$ with $(\uomega{\bar k}) \in \ball{\alpha_0}$,    we have 
	\begin{small}
\begin{equation*}	 \setlength\abovedisplayskip{3pt}
\setlength\abovedisplayshortskip{3pt}
\setlength\belowdisplayskip{3pt}
\setlength\belowdisplayshortskip{3pt}
\begin{split}
	H(\uomega{\bar k+1}) - H(\uomega{\bar k}) 
\geq \frac{c_3}{2}\sum^d_{j=1}\sum^R_{i=1}\|\mathbf u^{{\bar k}+1}_{j,i} - \mathbf u^{\bar k}_{j,i}\|^2+ \frac{c_3}{2}\|\boldsymbol{ \omega}^{\bar k+1}-\boldsymbol{ \omega}^{\bar k}\|^2.
\end{split}
\end{equation*}
\end{small}
\end{theorem}
\begin{remark}\label{rmk:1}
	
	1. The proof of Lemma \ref{lem:sufficient_dec_zero_e_in_a_ball} indicates that the above inequality cannot be extended to $k=\bar k+1,\bar k+2,\ldots$ directly, unless further information can be incorporated.
	
	2. It can be seen that the above inequality holds for all subsequences sufficiently close to $(\uomega{*})$.
	
3. The assumption is not strong in practice, as it is very often observed that $V^*_j$'s, $d-t+1\leq j\leq d$ have full column rank.
\end{remark}

Comparing with Theorem \ref{th:sufficient_dec_positive_epsilon}, the above results are weaker, due to the absence of $\epsilon_2$. Nevertheless, in what follows we show that the above results suffice  for the  convergence.



\subsection{Converging to a KKT point}
 First we consider $\epsilon_1,\epsilon_1\geq \epsilon_0>0$.
\begin{theorem}
	\label{th:sub_convergence_positive_e}
		Let $\{\mathbf u^k_{j,i},\omega^k_i  \}$ be generated by Algorithm \ref{alg:main} started from any initializer $(\mathbf u^0_{j,i},\omega^0_i )$ with $\epsilon_1,\epsilon_2\geq \epsilon_0>0$ for solving \eqref{prob:ortho_main_max} where $1\leq t\leq d$. Then every limit point  is a KKT point in the sense of \eqref{prob:kkt_main_max} or \eqref{prob:kkt_main_max_omega}.
\end{theorem}
\begin{proof}
Let $(\mathbf u^*_{j,i},\omega^*_i)$ be a limit point of $\{ \mathbf u^k_{j,i},\omega^k_i    \}$ and let $\{\mathbf u^{k_l} _{j,i},\omega^{k_l}_i  \}_{l=1}^\infty\rightarrow (\mathbf u^*_{j,i},\omega^*_i)$ be a convergent subsequence.  From Theorem \ref{th:sufficient_dec_positive_epsilon} and Proposition \ref{prop:nondecreasing} we see that $\{\uomega{k_l+1}  \}_{l=1}^\infty\rightarrow (\mathbf u^*_{j,i})$ as well.   Taking the limit into lines 4 and 5 of Algorithm \ref{alg:main} yields that for $i=1,\ldots,R$ and $j=1,\ldots,d-t$,
\begin{equation}
\label{eq:sec:sub_conv:100}			      \setlength\abovedisplayskip{2pt}
\setlength\abovedisplayshortskip{2pt}
\setlength\belowdisplayskip{2pt}
\setlength\belowdisplayshortskip{2pt}
\mathbf v^*_{j,i} \omega^*_i = \mathcal A\mathbf u^*_{1,1} \otimes\cdots\otimes \mathbf u^*_{j-1,i}\otimes \mathbf u^*_{j+1,i}\otimes \cdots \otimes \mathbf u^*_{d,R} \cdot \omega^*_i = (\bignorm{ \tilde{\mathbf v}^*_{j,i} }-\epsilon_1 )\mathbf u^*_{j,i},
\end{equation}
where $\omega^*_i = \sigma^*_i/\bignorm{\boldsymbol{ \sigma}^*}$ with $\sigma^*_i = \mathcal A\bigotimesuiterate{^*}$, and
 $\tilde{\mathbf v}^*_{j,i} = \mathbf v^*_{j,i}\omega^*_i + \epsilon_1\mathbf u^*_{j,i}$, as that of Algorithm \ref{alg:main}. Simple calculation shows that $ \bignorm{ \tilde{\mathbf v}^*_{j,i} }-\epsilon_1 = \sigma^*_i\omega^*_i$. We then consider $j=d-t+1,\ldots,d$. From the algorithm and Lemma \ref{lem:polar_max} we have
 \[			      \setlength\abovedisplayskip{3pt}
 \setlength\abovedisplayshortskip{3pt}
 \setlength\belowdisplayskip{3pt}
 \setlength\belowdisplayshortskip{3pt}
 \innerprod{U^{k_l+1}_j}{ \tilde V^{k_l+1}_j } - \innerprod{ X  }{ \tilde V^{k_l+1}_j } \geq 0 ,~\forall X^\top X = I.
 \]
 By using Theorem \ref{th:sufficient_dec_positive_epsilon} again and by taking limit into the above relation,  we obtain 
 \[ \setlength\abovedisplayskip{3pt}
 \setlength\abovedisplayshortskip{3pt}
 \setlength\belowdisplayskip{3pt}
 \setlength\belowdisplayshortskip{3pt}
  \innerprod{U^{*}_j}{ \tilde V^{*}_j } - \innerprod{ X  }{ \tilde V^{*}_j } \geq 0 ,~\forall X^\top X = I,
 \]
 where $\tilde V^{*}_j =  V^*_j\Omega^* + \epsilon_2U^*_j$, with $\Omega^*={\rm diag}(\omega^*_1,\ldots,\omega^*_R )\in\mathbb R^{R\times R}$. The above inequality together with Lemma \ref{lem:polar_max} shows that for each $j$ there exists a   matrix $H^*_j$ such that $\tilde V^*_j = U^*_jH^*_j$, namely,
 \[			      \setlength\abovedisplayskip{2pt}
 \setlength\abovedisplayshortskip{2pt}
 \setlength\belowdisplayskip{2pt}
 \setlength\belowdisplayshortskip{2pt}
 V^*_j \Omega^* = U^*_j( H^*_j - \epsilon_2 I ). 
 \]
 Note that when writing in vector-wise form,  this is exactly the third equality of \eqref{prob:kkt_main_max_omega}  with respect to $\mathbf u^*_{j,i}$, except the difference in the coefficients. Finally, it is easy to see that $\boldsymbol{ \omega}^*$ satisfies \eqref{prob:ortho_main_max_omega}.  As a result, every limit point $(\mathbf u^*_{j,i},\omega^*)$ is a KKT point of  \eqref{prob:kkt_main_max_omega}, and $(\mathbf u^*_{j,i})$ is a KKT point of problem \eqref{prob:ortho_main_max}.
\end{proof}

When $\epsilon_1=\epsilon_2=0$, based on Theorem \ref{th:sufficient_dec_epsilon0} and item 2 of Remark \ref{rmk:1}, the results are stated as follows. The proof is similar to that of Theorem \ref{th:sub_convergence_positive_e} and is omitted.
\begin{theorem}
	\label{th:sub_convergence_zero_e}
		Let $\{\mathbf u^k_{j,i},\omega^k_i  \}$ be generated by Algorithm \ref{alg:main} started from any initializer $(\mathbf u^0_{j,i},\omega^0_i )$ with $\epsilon_1=\epsilon_2=  0$  for solving \eqref{prob:ortho_main_max} where $1\leq t\leq d$. If there is a limit point $(\mathbf u^*_{j,i},\omega^*_i)$, such that $V^*_j$ has full column rank, $j=d-t+1,\ldots,d$,   then it is a KKT point in the sense of \eqref{prob:kkt_main_max} or \eqref{prob:kkt_main_max_omega}.	
\end{theorem}
\subsection{Global convergence}\label{sec:global_conv_results}
If \eqref{prob:kkt_main_max}  or \eqref{prob:kkt_main_max_omega} have finitely many solutions, then Theorems \ref{th:sufficient_dec_positive_epsilon} and \ref{th:sufficient_dec_epsilon0} together with \cite[Proposition 8.3.10]{facchinei2007finite} suffice to show that the whole sequence $\{\mathbf u^k_{j,i} \}$ converges to a single KKT point $(\mathbf u^*_{j,i})$. In fact, similar to \cite[Theorem 5.1]{wang2015orthogonal}, using tools from algebraic geometry one can also prove that for almost all tensors the KKT system \eqref{prob:kkt_main_max_omega} has finitely many solutions. Thus one has that for \emph{almost all} tensors, the sequence generated by Algorithm \ref{alg:main} globally converges to a KKT point. 

However, our goal is to show that for \emph{all} tensors the global convergence holds. The results are presented in Theorems \ref{th:global_conv_positive_e} and \ref{th:global_conv_zero_e} with proof   left to Sect. \ref{sec:Proof_global_covergence}. In particular, the proof of Theorem \ref{th:global_conv_zero_e} does not directly follow from the existing framework of convergence proof, which might be of independent interest.
\begin{theorem}[Global convergence when $\epsilon_i>0$, $i=1,2$]
	\label{th:global_conv_positive_e}
			Let $\{\mathbf u^k_{j,i}  ,\omega^k_i\}$ be generated by Algorithm \ref{alg:main} started from any initializer $(\mathbf u^0_{j,i} ,\omega^0_i)$ with $\epsilon_1,\epsilon_2\geq \epsilon_0>0$  for solving \eqref{prob:ortho_main_max} where $1\leq t\leq d$. Then the whole sequence $\{\mathbf u^k_{j,i} ,\omega^k_i\}_{k=1}^\infty$ converges to a single limit point $(\mathbf u^*_{j,i},\omega^*_i)$, which is a KKT point in the sense of \eqref{prob:kkt_main_max} or \eqref{prob:kkt_main_max_omega}.
\end{theorem}


\begin{theorem}
	\label{th:global_conv_zero_e}
	Let $\{\mathbf u^k_{j,i} ,\omega^k_i  \}$ be generated by Algorithm \ref{alg:main} started from any initializer $(\mathbf u^0_{j,i} ,\omega^0_i )$ with $\epsilon_1=\epsilon_2=  0$  for solving \eqref{prob:ortho_main_max} where $1\leq t\leq d$. If there is a limit point $(\mathbf u^*_{j,i} ,\omega^*_i)$, such that $V^*_j$'s have full column rank, $j=d-t+1,\ldots,d$,   then the whole sequence converges to $(\uomega{*})$ which   is a KKT point in the sense of \eqref{prob:kkt_main_max} or \eqref{prob:kkt_main_max_omega}.	
\end{theorem}

As mentioned in Remark \ref{rmk:1}, 	since   $V^*_j$ often has full column rank in practice, the assumption makes sense.

\paragraph{Comparisons with existing convergence results} \cite{chen2009tensor} proposed ALS for \eqref{prob:ortho_main_max} where $t=d$, and showed that if $U^*_j$'s are of full column rank and $(\mathbf u^*_{j,i})$ is separated from other KKT points, and in addition, $\{\mathbf u^k_{j,i} \}$ lies in a neighborhood of $(\mathbf u^*_{j,i})$, then global convergence holds. Clearly, these assumptions are stronger than those in Theorem \ref{th:global_conv_zero_e}. 
\cite{wang2015orthogonal,guan2019numerical} both established the global convergence only for \emph{almost all} tensors, where the results of \cite{wang2015orthogonal}   hold when $t=1$; although \cite{guan2019numerical} gave the results for $1\leq t\leq d$, they relied on the assumption that   certain matrices constructed from every limit point admit simple leading singular values, and the columnwisely orthonormal factors (such as $V_j$, $d-t+1\leq j\leq d$ in the algorithm) constructed from every limit point have full column rank. Thus comparing with existing results, we obtain the global convergence for \emph{all} tensors for any $1\leq t\leq d$, with weaker or   without assumptions.

\section{Proofs of Global Convergence}\label{sec:Proof_global_covergence}
\subsection{Proof of Theorem \ref{th:global_conv_positive_e}}\label{sec:proof:global_conv_positive_e}
To begin with, we first need some definitions from nonsmooth analysis. Denote ${\rm dom}f:=\{\mathbf x\in\mathbb R^n\mid f(\mathbf x)<+\infty \}$.
\begin{definition}[c.f. \cite{attouch2013convergence}]
	\label{def:subdifferential}
	For $\mathbf x\in {\rm dom}f$, the Fr\'{e}chet subdifferential, denoted as $\hat \partial f(\mathbf x)$, is the  set of vectors $\mathbf w\in\mathbb R^n$ satisfying 
\begin{equation}\label{eq:f_subdiff}			      \setlength\abovedisplayskip{2pt}
\setlength\abovedisplayshortskip{2pt}
\setlength\belowdisplayskip{2pt}
\setlength\belowdisplayshortskip{2pt}
\liminf _{\mathbf y \neq \mathbf  x \atop\mathbf  y \rightarrow\mathbf  x} \frac{f(\mathbf  y)-f(\mathbf  x)-\langle \mathbf w, \mathbf  y-\mathbf  x\rangle}{\|\mathbf  x-\mathbf  y\|}\geq 0.
\end{equation}
	The subdifferential of $f$ at $\mathbf x\in{\rm dom}f$, written $\partial f$,  is defined as
	\[ \setlength\abovedisplayskip{3pt}
	\setlength\abovedisplayshortskip{3pt}
	\setlength\belowdisplayskip{3pt}
	\setlength\belowdisplayshortskip{3pt}
	\partial f(\mathbf x):=\left\{\mathbf w \in \mathbb{R}^{n}: \exists \mathbf x^{k} \rightarrow \mathbf x, f\left(\mathbf x^{k}\right) \rightarrow f(\mathbf x), \mathbf w^{k} \in \hat{\partial} f\left(\mathbf x^{k}\right) \rightarrow \mathbf w\right\}.
	\]
\end{definition}
From the above definitions, it can be seen that $\hat\partial f(\mathbf x)\subset \partial f(\mathbf x)$ for each $\mathbf x\in\mathbb R^n$ \cite{bolte2014proximal}. When $f(\cdot)$ is differentiable, they collapse to the usual definition of gradient $\nabla f(\cdot)$. For a minimization problem $\min_{\mathbf x\in\mathbb R^n}f(\mathbf x)$, if $\mathbf x^*$ is a minimizer, then it holds that 
\[			      \setlength\abovedisplayskip{3pt}
\setlength\abovedisplayshortskip{3pt}
\setlength\belowdisplayskip{3pt}
\setlength\belowdisplayshortskip{3pt}
0\in \partial f(\mathbf x^*).
\]
Such a point is called a critical point of $f(\cdot)$.  The above relation generalizes the definitions of KKT point and   system to a wide range of optimization problems.

An extended-real-valued function is a function $f:\mathbb R^n\rightarrow [-\infty,\infty]$; it  is called proper if $f(\mathbf x)>-\infty$ for all $\mathbf x$ and $f(\mathbf x)<\infty$ for at least one $\mathbf x$. Such a function is called closed if it is lower semi-continuous (l.s.c. for short).
The key property to show the global convergence for a nonconvex algorithm is the Kurdyka-\L{}ojasiewicz (KL) property given as follows:
\begin{definition}[KL property and KL function, c.f.  \cite{bolte2014proximal,attouch2013convergence}]\label{def:kl} A proper function $f$ is said to have the KL property at $\overline{\mathbf x}\in {\rm dom}\partial f :=\{\mathbf x\in\mathbb R^n\mid \partial f(\mathbf x)\neq\emptyset  \}$, if there exist $\bar\epsilon\in(0,\infty]$, a neighborhood $\mathcal N$ of $\overline{\mathbf x}$, and a continuous and concave function   $\psi: [0,\bar\epsilon) \rightarrow \mathbb R_+$ which is continuously differentiable on $(0,\bar\epsilon)$ with positive derivatives and $\psi(0)=0$, such that for all $\mathbf x\in \mathcal N$ satisfying $f( \overline{\mathbf x}) <f({\mathbf x}) < f(\overline{\mathbf x}) + \bar\epsilon $, it holds that
	\begin{equation}\label{eq:kl_property}			      \setlength\abovedisplayskip{3pt}
	\setlength\abovedisplayshortskip{3pt}
	\setlength\belowdisplayskip{3pt}
	\setlength\belowdisplayshortskip{3pt}
	\psi^\prime(f(\mathbf x) - f(\overline{\mathbf x})  ){\rm dist}(0,\partial f(\mathbf x)) \geq 1,
	\end{equation}
where ${\rm dist}(0,\partial f(\mathbf x)) $ means the distance from the original point to the set  $\partial f(\mathbf x)$.	If a proper and l.s.c. function $f$ satisfies the KL property at each point of ${\rm dom}\partial f$, then $f$ is called a KL function. 
\end{definition}

It is known that KL functions are ubiquitous in applications \cite[p. 467]{bolte2014proximal}. As will be shown later, our problem can be seen as also a KL function. 

Let $\{\mathbf x^k  \}$ be generated by an algorithm for solving a minimization problem $\min_{\mathbf x\in\mathbb R^n} f(\mathbf x)$. \cite{attouch2013convergence} demonstrates that if $f(\cdot)$ is KL and certain conditions are met, then the whole sequence converges to a critical point. For fixed constants $a,b>0$, the conditions are presented as follows:

\noindent {\bf H1}.
$
f(\mathbf x^{k}) - f(\mathbf x^{k+1}) \geq a\|\mathbf x^k-\mathbf x^{k+1}\|^2,~\forall k;
$

\noindent {\bf H2}. There exists $\mathbf y^{k+1}\in \partial f(\mathbf x^{k+1})$ such that
\begin{equation*}			      \setlength\abovedisplayskip{3pt}
\setlength\abovedisplayshortskip{3pt}
\setlength\belowdisplayskip{3pt}
\setlength\belowdisplayshortskip{3pt}
\|\mathbf y^{k+1}\|\leq b\|\mathbf x^{k+1} - \mathbf x^k\|;
\end{equation*}

\noindent {\bf H3}. There exists a subsequence $\{\mathbf x^{k_j}  \}_{j=1}^\infty$ of $\{\mathbf x^k \}_{k=1}^\infty$ and $\tilde{\mathbf x}$ such that
\[			    
\mathbf x^{k_j}\rightarrow \tilde{\mathbf x},~{\rm and}~ f(\mathbf x^{k_j})\rightarrow f(\tilde{\mathbf x}),~j\rightarrow \infty.
\]
\begin{theorem}(\cite[Theorem 2.9]{attouch2013convergence})
	\label{th:kl}
	Let $f:\mathbf R^n\rightarrow R\cup \{+\infty\}$ be a proper l.s.c. function. If a sequence $\{\mathbf x^k \}_{k=1}^\infty$ satisfies {\bf H1}, {\bf H2} and {\bf H3}, and if $f(\cdot)$ has the KL property at $\tilde{\mathbf x}$ specified in {\bf H3},  then $\lim_{k\rightarrow \infty } \mathbf x^k = \tilde{\mathbf x}$; moreover, $\tilde{\mathbf x}$ is a critical point of $f(\cdot)$.	
\end{theorem}

Denote $I_C(\mathbf x)$ as the indicator (characteristic) function of a closed set $C\subset\mathbb R^n$ as
\[
I_C(\mathbf x) = 0,~{\rm if}~\mathbf x\in C;~ I_C(\mathbf x) = +\infty,~{\rm if}~\mathbf x\not\in C.
\]
Denote
 \begin{equation*}			      \setlength\abovedisplayskip{2pt}
 \setlength\abovedisplayshortskip{2pt}
 \setlength\belowdisplayskip{2pt}
 \setlength\belowdisplayshortskip{2pt}
 \begin{split}
&C_{j,i}:=\{\mathbf u_{j,i} \in\mathbb R^{n_j} \mid \|\mathbf u_{j,i}\|=1  \}, 1\leq j\leq d-t,1\leq i\leq R; \\
&C_{j}:=\{U_j \in\mathbb R^{n_j\times R} \mid U_j^\top U_j = I \},~d-t+1\leq j\leq d; \\
& C_{\boldsymbol{ \omega}} := \{ \boldsymbol{ \omega}\in\mathbb R^R\mid \|\boldsymbol{ \omega}\|=1 \}.
 \end{split}
 \end{equation*}
Then we consider the following minimization problem, as a variant of \eqref{prob:ortho_main_max_omega}  
	 \begin{eqnarray}
\label{prob:ortho_main_min_omega_variant}
&&\min~ J(\mathbf u_{j,i}) = J(\mathbf u_{j,i},\omega_i):=-H(\mathbf u_{i,j},\omega_i) -\frac{\epsilon_1}{2}\sum^{d-t}_{j=1}\sum^R_{i=1}\|\mathbf u_{j,i}\|^2 - \frac{\epsilon_2}{2}\sum^d_{j=d-t+1}\|U_j\|_F^2  \nonumber\\
&&~~~~~~~~~~~~~~~~~~~~~~+ \sum^{d-t}_{j=1} \sum^R_{i=1} \indicatorf{j,i}{\mathbf u_{j,i}} + \sum^d_{j=d-t+1}\indicatorf{j}{U_j} + \indicatorf{\boldsymbol{ \omega}}{\boldsymbol{ \omega}}.
\end{eqnarray}
Without the augmented terms $-\frac{\epsilon_1}{2}\sum^{d-t}_{j=1}\sum^R_{i=1}\|\mathbf u_{j,i}\|^2 - \frac{\epsilon_2}{2}\sum^d_{j=d-t+1}\|U_j\|_F^2$, the above problem is in fact the same as \eqref{prob:ortho_main_max_omega}, where the ``$\min$'' is converted to ``$\max$'', and the constraints are respectively replaced by their indicator functions. On the other hand, note that under the constraints, the augmented terms boil down exactly to a constant  $-  (d-t)R\epsilon_1/2- tR\epsilon_2/2$. Therefore, \eqref{prob:ortho_main_min_omega_variant} is essentially the same as \eqref{prob:ortho_main_max_omega}.

\begin{lemma}\label{lem:hpothesis_satisfied}
	Under the setting of Theorem \ref{th:global_conv_positive_e}, {\bf H1}, {\bf H2} and {\bf H3} are met by the sequence $\{\uomega{k} \}$.
\end{lemma}
\begin{proof}
 From the definition of $J(\cdot)$ and Theorem \ref{th:sufficient_dec_positive_epsilon}, we see that
 \[
 J(\mathbf u^k_{j,i},\omega^k_i) - J(\mathbf u^{k+1}_{j,i},\omega^{k+1}_i) \geq \frac{\epsilon_0}{2}\sum^d_{j=1}\sum^R_{i=1}\|\mathbf u^{k+1}_{j,i} - \mathbf u^k_{j,i}\|^2 + \frac{c}{2}\|\boldsymbol{ \omega}^{k+1}-\boldsymbol{ \omega}^k\|^2,~\forall k,
 \]
 and so {\bf H1} holds. On the other hand, for any limit point $(\uomega{*})$ of $\{\uomega{k} \}$ with $\{\uomega{k_l} \}_{l=1}^\infty\rightarrow (\uomega{*})$, clearly {\bf H3} holds. It remains to verify {\bf H2}.
 
  For $1\leq j\leq d-t$, $1\leq i\leq R$, recall that    
  $\mathbf v_{j,i}^{k+1} =  {\mathcal A}{\mathbf u_{1,i}^{k+1}\otimes \cdots\otimes \mathbf u_{j-1,i}^{k+1} \otimes \mathbf u_{j+1,i}^k \otimes\cdots\otimes \mathbf u_{d,i}^k  }$, and $\tilde{\mathbf v}^{k+1}_{j,i} = \mathbf v^{k+1}_{j,i}\omega^k + \epsilon_1\mathbf u^k_{j,i} $.    We first show that 
\begin{equation}\label{eq:proof:1}
			      \setlength\abovedisplayskip{3pt}
\setlength\abovedisplayshortskip{3pt}
\setlength\belowdisplayskip{3pt}
\setlength\belowdisplayshortskip{3pt}
 \tilde{\mathbf v}^{k+1}_{j,i} \in \partial \indicatorf{j,i}{\mathbf u^{k+1}_{j,i}}.
\end{equation}
 Since  by the discussions below Definition \ref{def:subdifferential}, $\hat\partial \indicatorf{j,i}{\mathbf u^{k+1}_{j,i}}\subset  \partial \indicatorf{j,i}{\mathbf u^{k+1}_{j,i}}$, it suffices to show $\tilde{\mathbf v}^{k+1}_{j,i} \in\hat\partial \indicatorf{j,i}{\mathbf u^{k+1}_{j,i}}$. From the definition of $\hat\partial \indicatorf{j,i}{\cdot}$, if $\mathbf y\not\in C_{j,i} $, then the left hand-side of \eqref{eq:f_subdiff} is infinity and the inequality of \eqref{eq:f_subdiff} with $\mathbf w= \tilde{\mathbf v}^{k+1}_{j,i}$ naturally holds; otherwise, when $\mathbf y\in C_{j,i}$, namely, $\|\mathbf y\|=1$, from the definition of $\mathbf u^{k+1}_{j,i}$, we have
 \[			      \setlength\abovedisplayskip{2pt}
 \setlength\abovedisplayshortskip{2pt}
 \setlength\belowdisplayskip{2pt}
 \setlength\belowdisplayshortskip{2pt}
 \mathbf u^{k+1}_{j,i} = \frac{  {\tilde{\mathbf v}^{k+1}_{j,i}}  }{\bignorm{\tilde{\mathbf v}^{k+1}_{j,i}}} = \arg\max_{ \|\mathbf y\|=1 } \innerprod{\mathbf y}{ \tilde{\mathbf v}^{k+1}_{j,i}  }\Leftrightarrow \langle \tilde{\mathbf v}^{k+1}_{j,i},\mathbf u^{k+1}_{j,i}-\mathbf y\rangle \geq 0,~\forall \|\mathbf y\|=1,
 \]
 which clearly shows 
 \[			      \setlength\abovedisplayskip{3pt}
 \setlength\abovedisplayshortskip{3pt}
 \setlength\belowdisplayskip{3pt}
 \setlength\belowdisplayshortskip{3pt}
 \liminf _{\mathbf y \neq \mathbf   u^{k+1}_{j,i}, \mathbf  y \rightarrow\mathbf   u^{k+1}_{j,i} } \frac{     \indicatorf{j,i}{\mathbf y}-  \indicatorf{j,i}{\mathbf u^{k+1}_{j,i}}  -\langle \tilde{\mathbf v}^{k+1}_{j,i}, \mathbf  y -\mathbf  u^{k+1}_{j,i} \rangle}{\|\mathbf  y-\mathbf   u^{k+1}_{j,i} \|}\geq 0,
 \]
 and so $\tilde{\mathbf v}^{k+1}_{j,i} \in\hat\partial \indicatorf{j,i}{\mathbf u^{k+1}_{j,i}}\subset  \partial \indicatorf{j,i}{\mathbf u^{k+1}_{j,i}}$.
 
 Next, we define \begin{equation}	\label{eq:proof:global:11}		      \setlength\abovedisplayskip{2pt}
 \setlength\abovedisplayshortskip{2pt}
 \setlength\belowdisplayskip{2pt}
 \setlength\belowdisplayshortskip{2pt}
 \hat{\mathbf v}^{k+1}_{j,i}  := \mathcal A\bigotimes_{l\neq j}\nolimits\mathbf u^{k+1}_{l,i}\omega^{k+1}_i + \epsilon_1\mathbf u^{k+1}_{j,i} .\end{equation}
 Then the subdifferential   of $J(\cdot)$ with respect to $\mathbf u_{j,i}$ at $(\uomega{k+1})$ is exactly 
 \[
\partial_{ \mathbf u_{j,i} } J(\mathbf u^{k+1}_{j,i},\omega^{k+1}_i)  = -\hat{\mathbf v}^{k+1}_{j,i} + \indicatorf{j,i}{\mathbf u^{k+1}_{j,i}},
 \]
 which together with \eqref{eq:proof:1} yields that
\begin{equation}
\label{eq:proof:2}
 \tilde{\mathbf v}^{k+1}_{j,i} - \hat{\mathbf v}^{k+1}_{j,i} \in \partial_{ \mathbf u_{j,i} } J(\mathbf u^{k+1}_{j,i},\omega^{k+1}_i), ~1\leq j\leq d-t,1\leq i\leq R. 
\end{equation}

We then consider $j=d-t+1,\ldots,d$ and $i=1,\ldots,R$. Recall that $\tilde V^{k+1}_j = V^{k+1}_j\Omega^k + \epsilon_2 U^k_j$ where $\Omega^k = {\rm diag}(\omega^k_1,\ldots,\omega^k_R)$. We first show that
\begin{equation}
\label{eq:proof:3}
\tilde V^{k+1}_j \in \partial \indicatorf{j}{U^{k+1}_j}.
\end{equation}
Similar to the above argument, when $Y\not\in C_j$, \eqref{eq:f_subdiff} with $\mathbf w = \tilde V^{k+1}_j$ naturally holds; when $Y\in C_j$, from the definition of $U^{k+1}_j$ and Lemma \ref{lem:polar_max},
\[
U^{k+1}_j \in\arg\max_{Y^\top Y=I}\innerprod{ Y}{\tilde V^{k+1_j}} \Leftrightarrow \innerprod{\tilde V^{k+1}_j}{ U^{k+1}_j - Y   }\geq 0,~\forall Y^\top Y=I,
\]
which again shows that \eqref{eq:f_subdiff} holds with $\mathbf w = \tilde V^{k+1}_j$. As a result, \eqref{eq:proof:3} is true. Next, we define $\hat V^{k+1}:= [ \hat{\mathbf v}^{k+1}_{j,1},\ldots,\hat{\mathbf v}^{k+1}_{j,R}     ]\in\mathbb R^{n_j\times R}$ where $\hat{\mathbf v}^{k+1}_{j,i} $ is given in \eqref{eq:proof:global:11}	. Therefore,
\[
\partial_{U_j} J(\mathbf u^{k+1}_{j,i},\omega^{k+1}_i) = - \hat V^{k+1}_j + \indicatorf{j}{U^{k+1}_j},
\]
which together with \eqref{eq:proof:3} shows that
\begin{equation}
\label{eq:proof:4}
\tilde V^{k+1}_j - \hat V^{k+1}_j \in \partial_{U_j} J(\mathbf u^{k+1}_{j,i},\omega^{k+1}_i ) ,~d-t+1\leq j\leq d.
\end{equation}

Finally, 
 in the same vein we can show that $\boldsymbol{ \sigma}^{k+1} \in \indicatorf{\boldsymbol{ \omega}}{\boldsymbol{ \omega}^{k+1}}$, which together with $\partial_{\boldsymbol{ \omega}} J(\mathbf u^{k+1}_{j,i}, { \omega}^{k+1}_i) = -\boldsymbol{ \sigma}^{k+1} + \indicatorf{\boldsymbol{ \omega}}{\boldsymbol{ \omega}^{k+1}} $ gives
\begin{equation}\label{eq:proof:5}
0\in \partial_{\boldsymbol{ \omega}}J(\mathbf u^{k+1}_{j,i},  { \omega}^{k+1}_i).
\end{equation}
Combining \eqref{eq:proof:2},  \eqref{eq:proof:4} and \eqref{eq:proof:5}, we see that 
\begin{equation}
\label{eq:proof:6}
( \tilde{\mathbf v}^{k+1}_{1,1} -\hat{\mathbf v}^{k+1}_{1,1} ,\ldots, \tilde{\mathbf v}^{k+1}_{d-t,R} -\hat{\mathbf v}^{k+1}_{d-t,R}, \tilde U^{k+1}_{d-t+1}-\hat U^{k+1}_{d-t+1},\ldots, \tilde U^{k+1}_d - \hat U^{k+1}_d, 0 ) \in \partial J( \mathbf u^{k+1}_{j,i},  { \omega}^{k+1}_i  ).
\end{equation}
The remaining task is to upper bound the left hand-side of \eqref{eq:proof:6} by $\bignorm{ (\mathbf u^{k+1}_{j,i},\omega^{k+1}_i  ) - (\mathbf u^k_{j,i},\omega^k_i)    }$. By noticing the definition of $\tilde{\mathbf v}_{j,i}$, $\hat{\mathbf v}_{j,i}$, $\tilde V_j$, $\hat V_j$, and by using Lemma \ref{lem:app1}, it is not hard to verify that there is a constant $b>0$ such that
\begin{eqnarray}
\label{eq:proof:7}
&&\bignorm{( \tilde{\mathbf v}^{k+1}_{1,1} -\hat{\mathbf v}^{k+1}_{1,1} ,\ldots, \tilde{\mathbf v}^{k+1}_{d-t,R} -\hat{\mathbf v}^{k+1}_{d-t,R}, \tilde U^{k+1}_{d-t+1}-\hat U^{k+1}_{d-t+1},\ldots, \tilde U^{k+1}_d - \hat U^{k+1}_d, 0 )}_F  \nonumber\\
&\leq& b\bignorm{ (\mathbf u^{k+1}_{j,i},\omega^{k+1}_i  ) - (\mathbf u^k_{j,i},\omega^k_i)    },
\end{eqnarray}
where $b$ only depends on $\epsilon_i$, $i=1,2$, $c$, $d$, $H^\infty$ defined in Proposition \ref{prop:nondecreasing}, and $\bignorm{\mathcal A}$ where $\bignorm{\mathcal A}$ denotes certain norm of $\mathcal A$. This verifies {\bf H2}. The proof has been completed.
\end{proof}

The verification of the following results uses some basic definitions and propositions concerning semi-algebraic sets and functions given in Appendix \ref{sec:app:kl}.
\begin{lemma}\label{lem:kl_J_positive_e}
For $\epsilon_i\geq 0$, $i=1,2$, $J(\cdot)$ is proper, l.s.c., and	admits the KL property   at any KKT point $(\mathbf u_{j,i},\omega_i)$ of \eqref{prob:kkt_main_max_omega}.
\end{lemma}
\begin{proof}
	Since $J(\cdot)$ is given by the sum of polynomial functions and indicator functions of   closed sets, $J(\cdot)$ is therefore proper and l.s.c..
	
	 On the other hand, the constrained sets in \eqref{prob:ortho_main_max_omega} are all Stiefel manifolds; then Proposition \ref{prop:semi_algebraic} (see also \cite[Example 2]{bolte2014proximal}) demonstrates that they are   semi-algebraic sets, and their indicator functions are semi-algebraic functions. Again, Proposition \ref{prop:semi_algebraic} shows that $J(\cdot)$, as the finite sum of semi-algebraic functions, is itself semi-algebraic as well. Lemma \ref{lem:semi-algebraic-fun} thus tells us that $J(\cdot)$ satisfies the KL property at any point of ${\rm dom}J =\{(\mathbf u_{j,i},\omega_i)\mid J(\mathbf u_{j,i},\omega_i) <+\infty \}$. It is clear that any KKT point of \eqref{prob:ortho_main_max_omega} is in ${\rm dom}J$. Thus the claim is true. 
\end{proof}

\begin{proof}[Proof of Theorem \ref{th:global_conv_positive_e}]
Lemma \ref{lem:hpothesis_satisfied} shows that the {\bf H1}, {\bf H2} and {\bf H3} are satisfied.  Theorem \ref{th:sub_convergence_positive_e} already proves that any limit point $(\uomega{*})$ is a KKT point of \eqref{prob:ortho_main_max_omega}. Thus by Lemma \ref{lem:kl_J_positive_e}, $J(\cdot)$ has the KL property at $(\mathbf u^*_{j,i},\omega^*_i)$. These together with Theorem \ref{th:kl} demonstrate  the assertion.
\end{proof}


\subsection{Proof of Theorem \ref{th:global_conv_zero_e}}\label{sec:proof:global_conv_zero_e}
In the case that $\epsilon_1=\epsilon_2=0$,   Theorem \ref{th:kl} cannot be applied anymore, because {\bf H1} in the previous subsection is not satisfied, due to that the inequality  in Theorem \ref{th:sufficient_dec_epsilon0} was only established locally.  Nevertheless, in this subsection we show that if  $(\uomega{k})$ is sufficiently close to $(\uomega{*})$, then together with the KL inequality \eqref{eq:kl_property}, all the points after $(\uomega{k})$ will converge  to $(\uomega{*})$.

When $\epsilon_1=\epsilon_2=0$, we redefine   $J_0(\cdot)$ and the problem as 
	 \begin{equation}
\label{prob:ortho_main_min_omega_variant_zero_e}
\begin{split}
&\min~ J_0(\mathbf u_{j,i}) = J_0(\mathbf u_{j,i},\omega_i):=-H(\mathbf u_{i,j},\omega_i)   \\
&~~~~~~~~~~~~~~~~~~~~~~+ \sum^{d-t}_{j=1}\sum^R_{i=1}\indicatorf{j,i}{\mathbf u_{j,i}} + \sum^d_{j=d-t+1}\indicatorf{j}{U_j} + \indicatorf{\boldsymbol{ \omega}}{\boldsymbol{ \omega}}.
\end{split}
\end{equation}

In what follows, we denote   $\Delta^{k,k+1}:=   (\mathbf u^{k}_{j,i},\omega^k_i   ) - (\mathbf u^{k+1}_{j,i},\omega^{k+1}_i  )    $.  Theorem \ref{th:sufficient_dec_epsilon0} is restated as follows by using the language of Lemma \ref{lem:sufficient_dec_zero_e_in_a_ball}.
\begin{lemma}[Restatement of Theorem \ref{th:sufficient_dec_epsilon0}]\label{lem:sufficient_dec_zero_e}
	Under the setting of Theorem \ref{th:sufficient_dec_epsilon0}, there exist a positive constant $c_3>0$ and an $\alpha_0>0$, such that  if   $(\uomega{\bar k}) \in \ball{\alpha_0}$,     then
		\[
					      \setlength\abovedisplayskip{2pt}
		\setlength\abovedisplayshortskip{2pt}
		\setlength\belowdisplayskip{2pt}
		\setlength\belowdisplayshortskip{2pt}
		J_0(\mathbf u^{\bar k},\omega^{\bar k }_i)- J_0(\mathbf u^{\bar k +1}  ,\omega^{\bar k +1}_i )\geq \frac{c_3}{2}\bignorm{\deltak{\bar k}{\bar k +1}}^2.
		\]
\end{lemma}
When $\epsilon_1=\epsilon_2=0$, define  $\hat{\mathbf u}^k_{j,i}$ and $\hat V_j^k$ similar to those in Lemma \ref{lem:hpothesis_satisfied}.       Using an analogous argument, we have
\begin{lemma}\label{lem:h2_zero_e}
	Assume that $\epsilon_1=\epsilon_2=0$. Then $(  {\mathbf v}^{k+1}_{1,1} -\hat{\mathbf v}^{k+1}_{1,1} ,\ldots,  {\mathbf v}^{k+1}_{d-t,R} -\hat{\mathbf v}^{k+1}_{d-t,R},   U^{k+1}_{d-t+1}-\hat U^{k+1}_{d-t+1},\ldots,   U^{k+1}_d - \hat U^{k+1}_d, 0 ) \in \partial J_0(\mathbf u^{k+1}_{k,i}, \omega^{k+1}_i )$, $\forall k$. 
There exist a positive constant $b_0>0$ such that for any $k$, 
\begin{small}
\begin{equation*}
\begin{split}
  \bignorm{(  {\mathbf v}^{k+1}_{1,1} -\hat{\mathbf v}^{k+1}_{1,1} ,\ldots,  {\mathbf v}^{k+1}_{d-t,R} -\hat{\mathbf v}^{k+1}_{d-t,R},   U^{k+1}_{d-t+1}-\hat U^{k+1}_{d-t+1},\ldots,   U^{k+1}_d - \hat U^{k+1}_d, 0 )} 
  \leq  b_0\bignorm{ \deltak{k}{k+1}   }.
\end{split}
\end{equation*}
\end{small}
 The above inequality implies that
 \[					      \setlength\abovedisplayskip{3pt}
 \setlength\abovedisplayshortskip{3pt}
 \setlength\belowdisplayskip{3pt}
 \setlength\belowdisplayshortskip{3pt}
 {\rm dist}(0, \partial J_0( \mathbf u^{k+1}_{j,i},\omega^{k+1}_i) ) \leq b_0 \bignorm{ \deltak{k}{k+1}   },~\forall k.
 \]
\end{lemma}
Recall that the setting of Theorem \ref{th:global_conv_zero_e} requires the existence of a limit point with $V^*_j$'s having full column rank, $d-t+1\leq j\leq d$.
The following inequality is due to the above two lemmas and the KL property \eqref{eq:kl_property}. Here we denote $J_0^k = J_0( \mathbf u^k_{j,i},\omega^k_i )$ and $J_0^\infty = J_0( \mathbf u^*_{j,i},\omega^*_i  )$ for notational simplicity.
\begin{lemma}\label{lem:decreasing_l_large}
Under the setting of Theorem \ref{th:global_conv_zero_e},  there is an $\alpha_0>0$, such that if $(\uomega{\bar k})\in \ball{\alpha_0}$,   then
	\begin{equation}\label{eq:proof:11}					      \setlength\abovedisplayskip{3pt}
	\setlength\abovedisplayshortskip{3pt}
	\setlength\belowdisplayskip{3pt}
	\setlength\belowdisplayshortskip{3pt}
\bignorm{ \deltak{\bar k}{\bar k+1}  } \leq \frac{1}{2}\bignorm{ \deltak{\bar k-1}{\bar k }   } + \frac{c_4}{c_3}\bigxiaokuohao{ \psi( J_0^{\bar k } - J_0^\infty  ) - \psi( J_0^{\bar k+1}  - J^\infty_0 )    },
	\end{equation}
	where $c_4\geq \frac{2b_0}{c_3}$, and $\psi$ is defined in Definition \ref{def:kl}.
\end{lemma}
\begin{proof}
	First, we  assume that $ (\mathbf u^{\bar k+1}_{j,i},\omega^{\bar k +1}_i) \neq (\mathbf u^{\bar k }_{j,i},\omega^{\bar k }_i)   $, otherwise  the inequality in question holds naturally. On the other hand, if  $J_0^{\bar k } = J_0^\infty$, then Proposition \ref{prop:nondecreasing} tells us that all the objective functions related to the points after $(\uomega{\bar k} )$ takes the same value. In particular, Lemma \ref{lem:sufficient_dec_zero_e} implies that    $(\uomega{\bar k+1}) =  (\uomega{\bar k}   )$. Using $(\uomega{\bar k+1}  ) $ in place of $ (\uomega{\bar k}   )$ and repeating the argument, we see that $ ( \uomega{\bar k}   ) =  (\uomega{\bar k+1}   ) =  ( \uomega{\bar k+2})   =\cdots =  ( \uomega{*} )  $, and \eqref{eq:proof:11} also holds.  Therefore,  we assume that $J_0^{\bar k}  > J^\infty_0$.

  Lemma \ref{lem:kl_J_positive_e} shows that $J_0(\cdot)$ has the KL property at $(\uomega{*})$.  
  From Definition \ref{def:kl} and tailored to our setting, there is an  $\alpha_1>0$, such that for all   $(\mathbf u_{j,i},\omega_i)\in \mathbb B_{\alpha_1}( \mathbf u^*_{j,i},\omega^*_i   ) \cap {\rm dom}J_0$, it holds that
\begin{equation}\label{eq:proof:9} \setlength\abovedisplayskip{3pt}
\setlength\abovedisplayshortskip{3pt}
\setlength\belowdisplayskip{3pt}
\setlength\belowdisplayshortskip{3pt}
  \psi^\prime( J_0(\mathbf u_{j,i},\omega_i   ) -J^\infty_0 ) {\rm dist}(0, \partial J_0(  \mathbf u_{j,i},\omega_i     )  ) \geq 1,
\end{equation}
  where $\psi$ is a continuously differentiable  and concave function with positive derivatives. Without loss of generality, we assume that $\alpha_0<\alpha_1$, where 	    $\alpha_0$  was given in  Lemma \ref{lem:sufficient_dec_zero_e_in_a_ball}. Thus if $(\uomega{\bar k})\in \ball{\alpha_0}$, then Lemma \ref{lem:sufficient_dec_zero_e} and \eqref{eq:proof:9} hold together.
  
 On the other hand, from the concavity of $\psi(\cdot)$ we have
\[ \setlength\abovedisplayskip{3pt}
\setlength\abovedisplayshortskip{3pt}
\setlength\belowdisplayskip{3pt}
\setlength\belowdisplayshortskip{3pt}
\psi( J_0^{\bar k  } - J_0^\infty  ) - \psi( J_0^{\bar k +1}  - J^\infty_0 ) \geq \psi^\prime(  J^{\bar k  }_0 - J^\infty_0  )(  J_0^{\bar k  } - J_0^{\bar k +1}      ),
\]
which together with \eqref{eq:proof:9} and Lemmas  \ref{lem:sufficient_dec_zero_e} and \ref{lem:h2_zero_e}  yields that when $(\uomega{\bar k})\in \ball{\alpha_0}$,
\begin{eqnarray*}
\label{eq:proof:10} \setlength\abovedisplayskip{3pt}
\setlength\abovedisplayshortskip{3pt}
\setlength\belowdisplayskip{3pt}
\setlength\belowdisplayshortskip{3pt}
\frac{c_3}{2}\bignorm{ \deltak{\bar k}{\bar k +1}  }^2 &\leq& J^{\bar k }_0 - J^{\bar k +1}_0 
 \leq  \frac{  \psi( J_0^{\bar k  } - J_0^\infty  ) - \psi( J_0^{\bar k +1}  - J^\infty_0 )  }{   \psi^\prime(  J^{\bar k }_0 - J^\infty_0  ) }\nonumber\\
&\leq& {\rm dist}(0, \partial J_0(  \mathbf u^{\bar k }_{j,i},\omega^{\bar k }_i )  ) \bigxiaokuohao{ \psi( J_0^{\bar k  } - J_0^\infty  ) - \psi( J_0^{\bar k +1}  - J^\infty_0 )    }\nonumber\\
&\leq& b_0 \bignorm{ \deltak{\bar k -1}{\bar k  }   } \bigxiaokuohao{ \psi( J_0^{\bar k  } - J_0^\infty  ) - \psi( J_0^{\bar k +1}  - J^\infty_0 )    }\nonumber\\
&=& b_0/c_4 \bignorm{ \deltak{\bar k-1}{\bar k  }   } \cdot c_4\bigxiaokuohao{ \psi( J_0^{\bar k } - J_0^\infty  ) - \psi( J_0^{\bar k+1 }  - J^\infty_0 )    },
\end{eqnarray*}
where $c_4$ is a constant large enough such that     $c_4\geq \frac{2b_0}{c_3}$ . Using $\sqrt{ab}\leq \frac{a+b}{2}$, we obtain
\begin{equation*} \setlength\abovedisplayskip{3pt}
\setlength\abovedisplayshortskip{3pt}
\setlength\belowdisplayskip{3pt}
\setlength\belowdisplayshortskip{3pt}
\bignorm{ \deltak{\bar k}{\bar k +1}  } \leq \frac{1}{2}\bignorm{ \deltak{\bar k -1}{\bar k }   } + \frac{c_4}{c_3}\bigxiaokuohao{ \psi( J_0^{\bar k  } - J_0^\infty  ) - \psi( J_0^{\bar k+1 }  - J^\infty_0 )    },
\end{equation*}
completing the proof.
\end{proof}

Now with the help of the above lemma, Lemma \ref{lem:sufficient_dec_zero_e_in_a_ball} and the KL property, by using the induction method we can extend the above lemma to a general sense. 
\begin{lemma}\label{lem:decreasing_k_large}
	Under the setting of Theorem \ref{th:global_conv_zero_e}, there is a large enough $\bar k$, such that when $k\geq \bar k$, 
	\begin{equation}\label{eq:proof:12}
				      \setlength\abovedisplayskip{2pt}
	\setlength\abovedisplayshortskip{2pt}
	\setlength\belowdisplayskip{2pt}
	\setlength\belowdisplayshortskip{2pt}
\bignorm{ \deltak{k}{k+1}  } \leq \frac{1}{2}\bignorm{ \deltak{k-1}{k}   } + \frac{c_4}{c_3}\bigxiaokuohao{ \psi( J_0^{k} - J_0^\infty  ) - \psi( J_0^{k+1}  - J^\infty_0 )    }.
	\end{equation}
\end{lemma}
\begin{proof}
Similar to the previous lemma, we can without loss of generality assume that $(\uomega{k+1}) \neq (\uomega{k})$, and $J^k_0> J^\infty_0$.
	
	Recall that in the proof of the previous lemma, if $(\uomega{\bar k})\in\ball{\alpha_0}$, then Lemma \ref{lem:sufficient_dec_zero_e}, Lemma \ref{lem:decreasing_l_large}, and KL inequality \eqref{eq:proof:9} hold. Let  $\alpha <  {\alpha_0} $. Thus there exists a sufficiently large $\bar k$, such that 
%
	\begin{equation}\label{eq:proof:13} \setlength\abovedisplayskip{3pt}
	\setlength\abovedisplayshortskip{3pt}
	\setlength\belowdisplayskip{3pt}
	\setlength\belowdisplayshortskip{3pt}
	(\uomega{\bar k})\in \mathbb B_{\alpha/4}(\uomega{*}),~ \psi(J^{\bar k }_0 - J^\infty_0  ) <\frac{\alpha c_3}{8c_4},~
	~\bignorm{ \deltak{\bar k -1}{\bar k } }<\frac{\alpha}{8}.
	\end{equation}
	In what follows, we use the  induction method to prove that for all $k\geq \bar k$, 1) \eqref{eq:proof:12} holds; 2) $\bignorm{\deltak{k}{*}}<\alpha$.
	
	Clearly, \eqref{eq:proof:12} holds when $k=\bar k $ due to Lemma \ref{lem:decreasing_l_large}, and \eqref{eq:proof:13} means that $\bignorm{\deltak{k}{*}}<\alpha$.
	
Now assume that the assertion holds for $k=\bar k,\ldots, K$, i.e., \eqref{eq:proof:12} and $\bignorm{ \deltak{k}{*} }< \alpha$ holds for $k=\bar k ,\ldots,K$. When $k=K+1$, 
\begin{small}
\begin{eqnarray}
\bignorm{\deltak{K+1}{*}} &\leq&  \sum_{k=\bar k}^{K}\nolimits\bignorm{ \deltak{k}{k+1} } + \bignorm{ \deltak{\bar k}{*} } \nonumber
\\
&\leq& \frac{1}{2}\sum^K_{k=\bar k}\nolimits\bignorm{ \deltak{k-1}{k }  } + \frac{c_4}{c_3}\bigxiaokuohao{   \psi( J^{\bar k}_0 - J^\infty_0  ) -  \psi( J^{K+1}_0 - J^\infty_0)       }+ \bignorm{ \deltak{\bar k}{*} }. \label{eq:proof:14}
\end{eqnarray}
\end{small}
Subtracting $\frac{1}{2}\sum^{K-1}_{k=\bar k}\bignorm{ \deltak{k}{k+1} }$ from both sides of   the second inequality yields
\[ 
\bignorm{ \deltak{K}{K+1}} + \frac{1}{2}\sum^{K-1}_{k=\bar k}\nolimits\bignorm{ \deltak{k}{k+1}  } \leq \bignorm{ \deltak{\bar k-1}{\bar k} } +  \frac{c_4}{c_3}\bigxiaokuohao{   \psi( J^{\bar k}_0 - J^\infty_0  ) -  \psi( J^{K+1}_0 - J^\infty_0)       } ,
\]
which together with \eqref{eq:proof:13} means that
\begin{equation}\label{eq:proof:15} 
\frac{1}{2}\sum^{K-1}_{k=\bar k}\nolimits\bignorm{ \deltak{k}{k+1} } \leq \frac{\alpha}{8} + \frac{\alpha}{8} = \frac{\alpha}{4}.
\end{equation}
This combining with  \eqref{eq:proof:14} gives that
\begin{eqnarray*} \setlength\abovedisplayskip{3pt}
	\setlength\abovedisplayshortskip{3pt}
	\setlength\belowdisplayskip{3pt}
	\setlength\belowdisplayshortskip{3pt}
\bignorm{\deltak{K+1}{*}} &\leq & \frac{1}{2}\sum^{K-1}_{k=\bar k}\bignorm{\deltak{k}{k+1}} + \bignorm{\deltak{\bar k-1}{\bar k}} + \frac{c_4}{c_3}\bigxiaokuohao{   \psi( J^{\bar k}_0 - J^\infty_0  ) -  \psi( J^{K+1}_0 - J^\infty_0)       } + \bignorm{ \deltak{\bar k}{*} }\\
&\leq& \frac{\alpha}{4} + \frac{\alpha}{8} + \frac{\alpha}{8} +\frac{\alpha}{4}< \alpha.
\end{eqnarray*}
From the definition of $\alpha$, we have that $(\uomega{K+1}) \in\ball{\alpha_0}$, and so from Lemma \ref{lem:decreasing_l_large}, it holds that
\[
\bignorm{ \deltak{K+1}{K+2}  } \leq \frac{1}{2}\bignorm{ \deltak{K}{K+1}   } + \frac{c_4}{c_3}\bigxiaokuohao{ \psi( J_0^{K+1} - J_0^\infty  ) - \psi( J_0^{K+2}  - J^\infty_0 )    }.
\]
As a consequence, induction method shows that \eqref{eq:proof:12} and $\bignorm{\deltak{k}{*}}<\alpha$ for all $k=\bar k,\bar k+1,\ldots$. This completes the proof.
\end{proof}

\begin{remark}
The proof cannot be  recognized as a simple application of  \cite{bolte2014proximal,attouch2013convergence}, due to   that Lemma \ref{lem:sufficient_dec_zero_e} only holds near $(\uomega{*})$, which requires us to prove the existence of $\ball{\alpha}$ such that the iterates after $(\uomega{\bar k})$ cannot go outside the ball, so that Lemmas \ref{lem:sufficient_dec_zero_e} and \ref{lem:decreasing_l_large} still valid. This  differs our proof from that of \cite{bolte2014proximal,attouch2013convergence}. 
\end{remark}

\begin{proof}[Proof of Theorem \ref{th:global_conv_zero_e}] Lemma \ref{lem:decreasing_k_large} already demonstrates the existence of $\bar k$ such that $\sum^\infty_{k=\bar k}\bignorm{ \deltak{k}{k+1}  }<+\infty$, namely, $\{\uomega{k} \}$ is a Cauchy sequence. Thus $\lim_{k\rightarrow\infty} (\uomega{k}) = (\uomega{*})$ follows directly. This together with Theorem \ref{th:sub_convergence_zero_e} gives the desired results.
\end{proof}

\section{Numerical Experiments}\label{sec:numer}
We evaluate the performance of Algorithm \ref{alg:main} with different $\epsilon_i$, and compare the performance of Algorithm \ref{alg:main} initialized by Procedure \ref{proc:init} and by random initializers.
All the   computations are conducted on an Intel i7-7770 CPU desktop computer with 32 GB of RAM. The supporting software is Matlab R2015b.  The Matlab  package Tensorlab  \cite{tensorlab2013} is employed for tensor operations.  The Matlab code of $\epsilon$-ALS and the initialization procedure is available   at \url{https://github.com/yuningyang19/epsilon-ALS}.

The tensors are generated as $\mathcal A = \mathcal B/\bignorm{\mathcal B} + \beta\cdot\mathcal N/\|\mathcal N\|$, where $\mathcal B = \sum^R_{i=1} \sigma_i\bigotimesu $, $\mathcal N$ is an unstructured tensor, and $\beta$ denotes the noise level, similar to \cite{sorensen2012canonical}. Among all experiments we set $\beta=0.1$. Here   $U_j$ and $\mathcal N$ are randomly drawn from a uniformly distribution in $[-1,1]$. The last $t$ $U_j$ are then made to be columnwisely orthonormal, while the first $(d-t)$ ones are columnwisely normalized.
The stopping criterion is 
$ \|(\mathbf u_{j,i}^{k+1}  )  -(\mathbf u_{j,i}^k ) \|/ \|(\mathbf u_{j,i}^k)\|  \leq 10^{-4}$ 
or $k\geq 2000$.

The update  of $\mathbf u_{j,i}$,  as having been commented in lines 4 and 10 of Algorithm \ref{alg:main},  can be done in parallel. In fact, these lines are equivalent to
\[
\tilde V^{k+1}_j = A_{(j)}\bigxiaokuohao{ U^{k+1}_1\odot \cdots \odot U^{k+1}_{j-1}\odot U^k_{j+1}\odot \cdots \odot U^k_d  }\cdot\Omega^k+\epsilon_i U^k_j,
\]
so as the inner ``for'' loop can be avoided. Here $A_{(j)} \in \mathbb R^{n_j \times \prod^d_{j^\prime\neq j   } n_{j\prime}    }$ is the mode-$j$ unfolding of $\mathcal A$, and $\odot$ denotes the Khatri-Rao product  \cite{kolda2010tensor}.  The \texttt{polar\_decomp} procedure in the algorithm is implemented based on the Matlab built-in function \texttt{svd}.

\begin{table}[h]
	\centering
	\caption{\label{tab:numer1}Performance of   $\epsilon$-ALS with different $\epsilon_i$. $n_j\geq R$, $\forall j$; $\beta=0.1$.}
	\setlength{\tabcolsep}{1.2mm}			
	\begin{mytabular}{llrrr|rrr|rrr|rrr}
		\toprule
		&     &  \multicolumn{3}{c}{$\epsilon_i=0$ }  &   \multicolumn{3}{c}{$\epsilon_i=10^{-8}$ }  &   \multicolumn{3}{c}{$\epsilon_i=10^{-6}$ } &   \multicolumn{3}{c}{$\epsilon_i=10^{-4}$ }  \\
		\toprule
		$t$      &      $(n_1,\ldots,n_4;R)$   & \multicolumn{1}{l}{Iter} & \multicolumn{1}{l}{time} & \multicolumn{1}{l}{rel.err} & \multicolumn{1}{l}{Iter} & \multicolumn{1}{l}{time} & \multicolumn{1}{l}{rel.err}  & \multicolumn{1}{l}{Iter} & \multicolumn{1}{l}{time} & \multicolumn{1}{l}{rel.err} & \multicolumn{1}{l}{Iter} & \multicolumn{1}{l}{time} & \multicolumn{1}{l}{rel.err}  \\
		\toprule
		1   & (10,10,10,10;5) & 6.2   & 0.03  & 3.80E-02  & 6.2   & 0.03 & 3.80E-02  & 6.2   & 0.03 & 3.80E-02   & 9.4  & 0.03  & 3.80E-02  \\
		1   & (10,10,10,30;10) & 14.6  & 0.07 & 5.65E-02  & 14.6 & 0.07 & 5.65E-02   & 15.0 & 0.07  & 5.65E-02   & 44.8  & 0.11 & 5.64E-02  \\
		1   & (20,20,20,20;10) & 14.0 & 0.09  & 4.68E-02  & 14.0 & 0.09  & 4.68E-02  & 15.6  & 0.10 & 4.68E-02  & 56.4 & 0.19 & 4.69E-02   \\
		1   & (20,20,20,30;10) & 8.1  & 0.09  & 4.60E-02  & 8.1  & 0.09 & 4.60E-02   & 9.0 & 0.09   & 4.60E-02  & 59.2  & 0.26 & 4.60E-02  \\
		1   & (30,30,30,30;10) & 10.2 & 0.21  & 3.40E-02  & 10.3 & 0.19 & 3.40E-02   & 13.3 & 0.22   & 3.40E-02 & 67.8   & 0.66 & 3.41E-02  \\
		1   & (30,30,30,40;10) & 10.9  & 0.27 & 3.51E-02  & 10.9  & 0.26  & 3.51E-02  & 15.0 & 0.30  & 3.51E-02  & 98.8  & 1.21 & 3.52E-02  \\
		1   & (40,40,40,40;10) & 7.9   & 0.61  & 2.17E-02 & 7.9   & 0.57 & 2.17E-02  & 10.9  & 0.64 & 2.21E-02  & 47.3  & 1.55  & 2.17E-02 \\
		\midrule
		2   & (10,10,10,10;5) & 4.3   & 0.02 & 4.07E-02  & 4.3  & 0.02  & 4.07E-02  & 4.3   & 0.02 & 4.07E-02  & 6.6   & 0.02 & 4.07E-02  \\
		2   & (10,10,20,20;10) & 8.2   & 0.04 & 5.62E-02  & 8.2    & 0.04  & 5.62E-02& 8.6    & 0.04 & 5.62E-02 & 41.7  & 0.10 & 5.77E-02  \\
		2   & (20,20,20,20;10) & 7.9   & 0.05  & 4.81E-02 & 7.9  & 0.05  & 4.81E-02  & 8.7   & 0.05 & 4.81E-02  & 55.8   & 0.16  & 4.82E-02\\
		2   & (20,20,30,30;10) & 11.5  & 0.09 & 4.44E-02   & 12.4  & 0.09 & 4.43E-02  & 12.2  & 0.09 & 4.45E-02   & 135.1  & 0.57 & 4.42E-02  \\
		2   & (30,30,30,30;10) & 8.5 & 0.16   & 3.02E-02  & 8.5   & 0.16 & 3.02E-02  & 11.6 & 0.18  & 3.08E-02  & 140.2   & 1.20 & 3.10E-02 \\
		2   & (30,30,40,40;10) & 9.5   & 0.34 & 3.54E-02  & 9.6  & 0.34 & 3.54E-02   & 20.0  & 0.49 & 3.55E-02  & 116.0  & 1.82 & 3.52E-02  \\
		2   & (40,40,40,40;10) & 16.3  & 0.85 & 2.38E-02  & 10.9 & 0.70  & 2.44E-02  & 18.8  & 0.89 & 2.44E-02  & 87.7  & 2.59 & 2.45E-02  \\
		\midrule
		3   & (10,10,10,10;10) & 3.8   & 0.01  & 3.71E-02 & 3.8  & 0.01  & 3.71E-02  & 3.9   & 0.01 & 3.71E-02  & 6.5   & 0.01 & 3.71E-02  \\
		3   & (10,20,20,20;10) & 15.5  & 0.05& 8.80E-02   & 15.6  & 0.04  & 8.80E-02 & 16.9  & 0.05& 8.26E-02   & 130.5  & 0.29 & 8.44E-02  \\
		3   & (20,20,20,20;10) & 9.1   & 0.05  & 2.99E-02 & 9.1    & 0.04 & 2.99E-02 & 7.0  & 0.04  & 2.96E-02  & 48.5  & 0.14& 2.95E-02   \\
		3   & (20,30,30,30;10) & 4.2   & 0.09 & 1.94E-02  & 4.2   & 0.09 & 1.94E-02  & 5.1   & 0.09 & 1.94E-02  & 24.1  & 0.21& 1.93E-02   \\
		3   & (30,30,30,30;10) & 6.5   & 0.17 & 2.85E-02  & 6.5   & 0.17 & 2.85E-02  & 9.6   & 0.19 & 2.85E-02  & 123.2   & 1.17  & 2.85E-02\\
		3   & (30,40,40,40;10) & 12.1  & 0.62 & 4.80E-02  & 12.0 & 0.61   & 4.80E-02 & 22.2  & 0.81 & 4.80E-02  & 160.7 & 3.33 & 4.78E-02   \\
		3   & (40,40,40,40;10) & 2.9   & 0.73 & 1.14E-02  & 2.9   & 0.73 & 1.14E-02  & 3.3   & 0.74 & 1.14E-02  & 21.5  & 1.24 & 1.14E-02  \\
		\bottomrule
	\end{mytabular}%
\end{table}%
The evaluation of $\epsilon$-ALS initialized by Procedure \ref{proc:init} on different size of tensors are reported in Tables \ref{tab:numer1} and Table \ref{tab:numer2}. Here the size of the tensors in Table \ref{tab:numer1} satisfies $n_j\geq R$ for all $j$, while that in Table \ref{tab:numer2} is $n_{j_1} < R$, $1\leq {j_1} \leq d-t$, and $R\leq n_{j_2}$, $d-t+1\leq j_2\leq d$. 
The results are averaged over $50$ instances for each case. For convenience we set $\epsilon_1=\epsilon_2$, varying from $0$ to $10^{-4}$. In the table, 	`Iter' denotes the iterates, `time' represents the CPU time counting both Procedure \ref{proc:init} and Algorithm \ref{alg:main}, where the unit is second, 
and `rel.err' stands for the relative error between $U^{\rm out}_{j,i}$ generated by the algorithm and the true $U_{j,i}$ of $\mathcal A$. Due to permutation issues, `rel.err' is defined as
\[ \setlength\abovedisplayskip{3pt}
\setlength\abovedisplayshortskip{3pt}
\setlength\belowdisplayskip{3pt}
\setlength\belowdisplayshortskip{3pt}
{\rm rel.err} = \bignorm{(U_{j,i}) - (U^{\rm out}_{j,i}\cdot\Pi_{j,i})}_F/\bignorm{(U_{j,i}) }_F,
\]
where $\Pi_{j,i}=\arg\min_{\Pi\in\boldsymbol{\Pi}} \bignorm{U_{j,i} - U^{\rm out}_{j,i}\cdot \Pi}_F$ and $\boldsymbol{\Pi}$ denotes the set of permutation matrices; this follows \cite{sorensen2012canonical}. In the experiments, our first observation is that in all cases, the algorithm converges to a KKT point. From the tables, we observe that in terms of the iterates and CPU time, when $\epsilon_i\leq 10^{-6}$, the algorithm with a good initializer usually converges within a few iterates, demonstrating its efficient. The results also suggest us to choose a very small $\epsilon_i$. In terms of the relative error, the algorithm performs almost the same with different $\epsilon_i$, all of which are small, compared with the noise level. We do not present $t=4$ cases here, because from the structure of $\mathcal A$ and Procedure \ref{proc:init}, the initializer already yields a very high-quality solution to the problem, and $\epsilon$-ALS  usually stops within $2$ iterates. 

\begin{table}[htbp]
	\centering
	\caption{\label{tab:numer2}  Performance of   $\epsilon$-ALS with different $\epsilon_i$. $n_{j_1}<R\leq n_{j_2}$, $1\leq j_1\leq d-t, d-t+1\leq j_2\leq d$; $\beta=0.1$.}
	\setlength{\tabcolsep}{1.2mm}			
	\begin{mytabular}{ll|rrr|rrr|rrr|rrr}
		\toprule
		&     &  \multicolumn{3}{c}{$\epsilon_i=0$ }  &   \multicolumn{3}{c}{$\epsilon_i=10^{-8}$ }  &   \multicolumn{3}{c}{$\epsilon_i=10^{-6}$ } &   \multicolumn{3}{c}{$\epsilon_i=10^{-4}$ }  \\
		\toprule
		$t$      &      $(n_1,\ldots,n_4;)$   & \multicolumn{1}{l}{Iter} & \multicolumn{1}{l}{time} & \multicolumn{1}{l}{rel.err} & \multicolumn{1}{l}{Iter} & \multicolumn{1}{l}{time} & \multicolumn{1}{l}{rel.err}  & \multicolumn{1}{l}{Iter} & \multicolumn{1}{l}{time} & \multicolumn{1}{l}{rel.err} & \multicolumn{1}{l}{Iter} & \multicolumn{1}{l}{time} & \multicolumn{1}{l}{rel.err}  \\
		\toprule
		1     & (5,5,5,20;10) & 22.4  & 0.05  & 9.59E-02 & 22.4  & 0.05 & 9.59E-02  & 22.4  & 0.05 & 9.59E-02  & 27.2  & 0.06 & 9.59E-02  \\
		1     & (5,5,5,30;10) & 21.5  & 0.05  & 9.37E-02 & 21.5  & 0.05 & 9.37E-02  & 21.5  & 0.05 & 9.37E-02  & 25.8  & 0.05& 9.37E-02   \\
		1     & (10,10,10,40;20) & 40.4  & 0.16 & 9.48E-02  & 40.5  & 0.16 & 9.48E-02  & 42.3  & 0.17 & 9.49E-02  & 133.9 & 0.34 & 9.48E-02  \\
		\midrule
		2     & (5,5,30,30;20) & 19.5  & 0.07 & 1.01E-01  & 19.5 & 0.07  & 1.01E-01  & 20.3  & 0.07 & 1.01E-01  & 87.1  & 0.19& 1.01E-01   \\
		2     & (5,5,40,40;20) & 26.1  & 0.10 & 1.01E-01  & 26.1  & 0.09 & 1.01E-01  & 27.8  & 0.10& 1.01E-01   & 103.7  & 0.27  & 1.01E-01 \\
		2     & (10,10,40,40;30) & 28.2  & 0.22& 1.03E-01   & 28.3  & 0.22& 1.03E-01   & 36.0  & 0.25& 1.03E-01   & 304.2  & 1.45 & 1.02E-01  \\
		\midrule
		3     & (5,10,10,10;10) & 3.5   & 0.01& 3.11E-02   & 3.5   & 0.01 & 3.11E-02  & 3.5  & 0.01   & 3.11E-02 & 5.1   & 0.01 & 3.11E-02  \\
		3     & (5,20,20,20;10) & 11.2  & 0.03 & 7.66E-02  & 11.2  & 0.03& 7.66E-02   & 13.0  & 0.03& 7.58E-02   & 61.3  & 0.10  & 7.58E-02 \\
		3     & (5,40,40,40;20) & 35.7  & 0.26& 9.08E-02   & 35.5  & 0.26 & 9.08E-02  & 39.0  & 0.28& 8.79E-02   & 276.1  & 1.65& 8.54E-02   \\
		3     & (5,40,40,40;30) & 21.4  & 0.21 & 5.89E-02  & 19.1  & 0.19& 5.61E-02   & 25.4  & 0.24 & 5.62E-02  & 214.5  & 1.62 & 5.68E-02  \\
		3     & (10,20,20,20;15) & 5.9   & 0.03 & 3.62E-02  & 5.9   & 0.03 & 3.62E-02  & 6.5   & 0.03  & 3.62E-02 & 45.9  & 0.11 & 3.62E-02  \\
		\bottomrule
	\end{mytabular}%
\end{table}%

\begin{table}[htbp]
	\centering
	\caption{\label{tab:compare_proc_random}Comparisons of $\epsilon$-ALS initialized by Procedure \ref{proc:init} and by random initializers. $\epsilon_i=10^{-8}$; $\beta=0.1$.  }
	\begin{mytabular}{ll|rrr|rrr}
		\toprule
		&     &  \multicolumn{3}{c}{Procedure \ref{proc:init} }  &   \multicolumn{3}{c}{Random }   \\
		\toprule
		\multicolumn{1}{l}{$t$} &  $(n_1,\ldots,n_4;R)$      & \multicolumn{1}{l}{Iter} & \multicolumn{1}{l}{time} & \multicolumn{1}{l}{rel.err} & \multicolumn{1}{l}{Iter} & \multicolumn{1}{l}{time} & \multicolumn{1}{l}{rel.err} \\
		\toprule
		1     & (20,20,20,20;10) & \textbf{19.5} & \textbf{0.11} & \textbf{7.21E-02} & 161.6  & 0.39  & 4.69E-01 \\
		1     & (30,30,30,30;10) & \textbf{14.2} & \textbf{0.27} & \textbf{5.05E-02} & 114.7  & 0.97  & 5.04E-01 \\
		1     & (40,40,40,40;10) &\textbf{4.6} & \textbf{0.51} & \textbf{1.46E-02} & 146.8  & 3.70  & 5.36E-01 \\
		\midrule
		2     & (10,10,40,40;30) & \textbf{21.5} & 0.15  & \textbf{6.50E-02} & 32.7  & \textbf{0.14} & 8.08E-02 \\
		2     & (20,20,20,20;10) & \textbf{10.2} & 0.06  & \textbf{5.34E-02} & 18.8  & \textbf{0.04} & 7.65E-02 \\
		2     & (30,30,30,30;5) & \textbf{6.9} & 0.12  & \textbf{2.82E-02} & 14.5  & \textbf{0.11} & 3.73E-02 \\
		2     & (30,30,30,30;20) & \textbf{24.6} & \textbf{0.43} & \textbf{6.43E-02} & 52.4  & 0.64  & 8.81E-02 \\
		2     & (40,40,40,40;10) & \textbf{12.4} & \textbf{0.76} & \textbf{2.39E-02} & 48.0  & 1.39  & 7.61E-02 \\
		\midrule
		3     & (5,30,30,30;10) & \textbf{8.7} & \textbf{0.04} & \textbf{3.52E-02} & 15.4  & 0.05  & 4.72E-02 \\
		3     & (5,40,40,40;20) & \textbf{15.8} & \textbf{0.15} & \textbf{6.15E-02} & 26.9  & 0.18  & 8.23E-02 \\
		3     & (20,20,20,20;10) & \textbf{8.2} & \textbf{0.03} & \textbf{4.26E-02} & 15.1  & 0.04  & 6.22E-02 \\
		3     & (30,30,30,30;10) & \textbf{3.8} & \textbf{0.14} & \textbf{1.64E-02} & 15.0  & \textbf{0.14} & 3.15E-02 \\
		3     & (40,40,40,40;10) & \textbf{11.6} & 0.85  & \textbf{3.33E-02} & 31.3  & \textbf{0.83} & 7.71E-02 \\
		\bottomrule
	\end{mytabular}%
	\label{tab:addlabel}%
\end{table}%

We then compare $\epsilon$-ALS  initialized by Procedure \ref{proc:init} and by random initializers. We set $\epsilon_i=10^{-8}$. The results are reported in Table \ref{tab:compare_proc_random}, from which we   see that  the algorithm armed with Procedure \ref{proc:init} enjoys a better efficiency in most cases (here the time of Procedure \ref{proc:init} has also been counted); the iterates are much less than that with a random initializer when $t=1$. Considering the relative error, the former also performs much better. These results demonstrates the efficiency and effectiveness of $\epsilon$-ALS  initialized by Procedure \ref{proc:init}. 

\section{Conclusions}\label{sec:conclusions}
To overcome the global convergence issues raised in \cite{wang2015orthogonal} concerning ALS for solving orthogonal low rank tensor approximation problems where the columnwisely orthonormal factors are one or more than one, the   $\epsilon$-ALS is developed and analyzed in this work. The global convergence has been established for  {all tensors}  without any assumption. To accomplish the study, we also prove the global convergence of (unperturbed) ALS, if there exists a limit point in which the columnwise orthonormality factor matrices have full column rank; such an assumption makes sense and is not stronger than those in the literature. By combining the ideas of HOSVD and approximation solutions for tensor best rank-1 approximation, an initialization procedure is proposed, armed with which the   $\epsilon$-ALS exhibits a promising performance both in efficiency and effectiveness. Possible future work are: 1) to establish the linear convergence of the algorithm; 2) to give a theoretical lower bound of Procedure \ref{proc:init} when $t\geq 2$, or to design   procedures based on other types of HOSVD \cite{vannieuwenhoven2012new,grasedyck2010hierarchical}; 3) to automatically choose $R$
for the algorithm.

 {\scriptsize\section*{Acknowledgement}  This work was supported by NSFC Grant 11801100. }

   \bibliography{tensor,TensorCompletion,orth_tensor}
  \bibliographystyle{siam}

   \appendix

 \section{Auxiliary Results}\label{sec:app:kl}
 
   \begin{definition}[Semi-algebraic set  and function, see e.g., \cite{bochnak1998real} ]\label{def:semi_algebraic}
   A set $C \subset \mathbb R^n$ is called semi-algebraic, if there exists a finite number of real polynomials $p_{ij},q_{ij} : \mathbb R^n\rightarrow \mathbb R$, $1\leq i\leq s, 1\leq j\leq t$ such that
   \begin{small}
 \[
   C =  \bigcup_{i=1}^s\nolimits\bigcap_{j=1}^t\nolimits\{ \mathbf x \in \mathbb R^n ~|~  p_{ij}(\mathbf x)=0,~q_{ij}(\mathbf x) > 0  \}.
 \]
   \end{small}
   An extended real-valued function $F$ is said to be semi-algebraic if its graph
   \begin{displaymath}
   {\rm graph}\,F :=\left \{ (\mathbf x,t)\in \mathbb R^{n+1}~|~ F(\mathbf x)=t   \right\}
   \end{displaymath}
   is a semi-algebraic set.
   \end{definition}

\begin{proposition}\label{prop:semi_algebraic}(\cite[Example 2]{bolte2014proximal})
Stiefel manifolds are semi-algebraic sets. 
The following are examples of  semi-algebraic functions:
     \begin{enumerate}
     \item Real polynomial functions.
     \item Indicator functions of semi-algebraic sets.
     \item Finite sums, product and composition of semi-algebraic functions.
     \end{enumerate}
 
 \begin{lemma}(\cite[Theorem 3]{bolte2014proximal})
 	\label{lem:semi-algebraic-fun}
 Let $f(\cdot)$ be a proper and l.s.c. function. If $f(\cdot)$ is semi-algebraic, then it satisfies the KL property at any point of ${\rm dom}f$.
 \end{lemma}
 
\end{proposition}

  \begin{lemma}\label{lem:app1}
 For any integer   $d$ and any unit length vector $ \mathbf x_{j},\mathbf y_j$, there holds
 \begin{small}
 \[
\bignorm{\bigotimesx - \bigotimesy} \leq  \sum^d_{i=1}\nolimits  \|  \mathbf x_j - \mathbf y_j\|  . 
 \]
 \end{small}
  \end{lemma}
  \begin{proof}
      We apply the induction method on $d$. When $d=1$, the above inequality is true. Suppose it holds when $d=m$. When $d=m+1$,
      \begin{small}
      \begin{eqnarray*}
       \bignorm{  \bigotimes^{m+1}_{j=1}\nolimits \mathbf x_j -  \bigotimes^{m+1}_{j=1}\nolimits\mathbf y_j   }    
   &\leq & \bignorm{ \bigotimes^{m+1}_{j=1}\nolimits\mathbf x_j -  \bigotimes^m_{j=1}\nolimits\mathbf x_j\otimes \mathbf y_{m+1}    } + \bignorm{  \bigotimes^m_{j=1}\nolimits\mathbf x_j \otimes \mathbf y_{m+1}   -\bigotimes^{m+1}_{j=1}\nolimits\mathbf y_j   }     \\
       &= & \bignorm{ \mathbf x_{m+1} - \mathbf y_{m+1} } + \bignorm{ \bigotimes^m_{j=1}\nolimits\mathbf x_j -\bigotimes^m_{j=1}\nolimits\mathbf y_j  } \\
&\leq& \sum^{m+1}_{j=1}\nolimits\bignorm{ \mathbf x_j - \mathbf y_j },
      \end{eqnarray*}
  \end{small}
  where the equality follows from $\bignorm{ \mathbf x\otimes ( \mathbf y-\mathbf z)  } = \bignorm{ \mathbf y-\mathbf z}$ for $\|\mathbf x\|=1$. This completes the proof.
  \end{proof}

\end{document}